\documentclass[11pt]{amsart}
\usepackage{a4}
\usepackage{amssymb}
\makeatletter
\newcommand{\leqnomode}{\tagsleft@true}
\newcommand{\reqnomode}{\tagsleft@false}
\makeatother

\newtheorem{theorem}{Theorem}[section]
\theoremstyle{plain}
\newtheorem{lemma}[theorem]{Lemma}
\newtheorem{proposition}[theorem]{Proposition}

\newtheorem{fact}[theorem]{Fact}
\newtheorem{corollary}[theorem]{Corollary}

\theoremstyle{remark}

\newtheorem{definition}[theorem]{Definition}
\newtheorem{example}[theorem]{Example}

\newtheorem{notamark}[theorem]{Notation \& Remark}

\newtheorem{remark}[theorem]{Remark}

\numberwithin{equation}{section}

\newcommand{\ds}{\displaystyle}
\newcommand{\ph}{\varphi}
\newcommand{\eps}{\varepsilon}
\newcommand\trnorm[1]{\left\|\kern-1.2pt\left|#1\right\|\kern-1.2pt\right|}

\begin{document}
\title[]{Abstract Lorentz spaces and K\"othe duality}
\date{January 31, 2019}


\author[Kami\'nska]{Anna Kami\'{n}ska}
\address[Kami\'nska]{Department of Mathematical Sciences,
The University of Memphis, TN 38152-3240, U.S.A.}
\email{kaminska@memphis.edu}

\author[Raynaud]{Yves Raynaud}
\address[Raynaud]{CNRS, Sorbonne Universit\'e, Universit\'e Paris Diderot, Institut de Math\'ematiques de Jussieu-Paris Rive Gauche, IMJ-PRG, F-75005, Paris, France.}
\email{\texttt{yves.raynaud@upmc.fr}}

\maketitle
\bigskip

\begin{abstract}   
Given a symmetric Banach function space $E$ and a decreasing positive weight $w$ on $I=(0,a)$, $0<a\le \infty$, the  generalized Lorentz space $\Lambda_{E,w}$ is defined as the symmetrization of the canonical copy $E_w$ of $E$ on the measure space associated with the weight.   A class of functions $M_{E,w}$ is similarly defined in the spirit of Marcinkiewicz spaces as the symmetrization of the space $w\, E_w$.  Differently from the Lorentz space, which is a Banach function space, the class $M_{E,w}$ does not need to be even a linear space; but we show that if the weight $w$ is regular then this class is normable. 
Let also $Q_{E,w}$ be the smallest fully symmetric Banach function space  containing $M_{E,w}$. 
The K\"othe duality  of these classes is developed  here. The K\"othe dual of the class $M_{E,w}$ is identified as the Lorentz space $\Lambda_{E',w}$, while the K\"othe dual of  $\Lambda_{E,w}$ is $Q_{E',w}$. Several characterizations of  $Q_{E,w}$ are obtained, one of them states that a function belongs to $Q_{E,w}$ if and only if its level function in Halperin's sense with respect to $w$, belongs to $M_{E,w}$. The other characterizations are by optimization with respect to the Hardy-Littlewood submajorization order. These results are applied to a number of concrete Banach function spaces. In particular a  new description of the K\"othe dual space is provided for the classical Lorentz space $\Lambda_{p,w}$ and for the Orlicz-Lorentz space $\Lambda_{\varphi,w}$, which correspond respectively to the cases $E=L_p$ and $E=L_\ph$. 
\end{abstract}

\bigskip

Given a positive locally integrable weight $w$ on an interval $I=(0,a)$, $0<a\le \infty$, and $p\in [1,\infty)$, the classical Lorentz space $\Lambda_{p,w}$ is the set of  measurable functions $f$ having a non-increasing rearrangement $f^*$ such that $\int_I (f^*)^pw\,dm <\infty$, where  $m$ denotes the Lebesgue measure. This class is a symmetric Banach function space and the formula $\|f\|_{p,w}= (\int_I (f^*)^pw\,dm <\infty)^{1/p}$ defines a norm if and only if the weight $w$ is non-increasing \cite{BS, KPS}. Orlicz-Lorentz spaces may be defined in a similar way. Following \cite{K}, given an Orlicz function $\ph$, consider the  modular $\Phi$ defined on the set of Lebesgue measurable functions $L^0(I)$ by
\[\Phi(f) = \int_I \ph(f^*)w\,dm.\]
Then the Orlicz-Lorentz space $\Lambda_{\ph,w}$ is the set of  $f\in L^0(I)$ such that $\{c: \Phi(f/c)<\infty\}\ne\emptyset$.
If $w$ is non-increasing then $\Phi$ is convex, $\Lambda_{\ph,w}$ is a linear subset and an  ideal in $L^0(I)$, and the formula $\|f\|_{\ph,w}:=\inf\{c: \Phi(f/c)\le 1\}$ defines a norm, called the Luxemburg or second Nakano norm, for which $\Lambda_{\ph,w}$ is complete and symmetric. Clearly if $\ph(t)=t^p$ we recover the classical Lorentz space 
 $\Lambda_{p,w}$, so that Orlicz-Lorentz spaces are a generalization of ordinary Lorentz spaces. We refer to \cite{KR, KLR} for a study of K\"othe duality of Orlicz-Lorentz spaces. 
 
 Our goal in this paper is to generalize further the class of  Orlicz-Lorentz spaces by replacing Orlicz spaces by general symmetric K\"othe function spaces. We will use the fact that  the classical Lorentz spaces $\Lambda_{p,w}$ and the Orlicz-Lorentz spaces $\Lambda_{\ph,w}$ respectively, are symmetrizations \cite{KR09} with respect to the Lebesgue measure  of  the spaces $L_p$ and Orlicz spaces $L_\ph$ on the measure space $(I, wdm)$, respectively. Clearly the latter spaces are weighted $L_p(I,wdm)$ and $L_\varphi(I,wdm)$ spaces  respectively, and thus they are symmetric with respect to the measure $wdm$ on $I$.
 Thus it  is natural to consider symmetrizations, with respect to $m$, of arbitrary  K\"othe function spaces $\mathcal E$ which are symmetric with respect to the measure $wdm$. However, since we do not want the space parameter,  which will play the role of the exponent $p$ or the Orlicz function $\ph$, to depend on the weight $w$,  we choose to take as space $\mathcal E$ the natural copy $E_w$ on the measure space $(I,wdm)$ of a symmetric K\"othe function space $E$ defined on the measure space $(J,m)$ where $J$ is an interval $(0,b)$, $b\in (0,\infty]$.   The symmetrization of $E_w$ will thus be denoted by $\Lambda_{E,w}$ and called a {\it generalized Lorentz space}.
 
In this paper $w$ will always denote a non-increasing weight. We assume also that $E$ is {\it fully symmetric}  that is $E$ is hereditary by Hardy-Littlewood submajorization and the norm is monotone with respect to this submajorization. It follows from these hypotheses that the set $\Lambda_{E,w}$ is a linear space and that the formula $\|f\|_{\Lambda_{E,w}}=\|f^*\|_{E_w}$ defines a norm on it. In fact $\Lambda_{E,w}$ is a fully symmetric  Banach function space.  

 The goal is to provide a description of the K\"othe dual space of the Lorentz space $\Lambda_{E,w}$ in this abstract form, following the pattern of our previous article \cite{KR} on Orlicz-Lorentz spaces.  It was proved in \cite{HKM} that when the weight $w$ is {\it regular} the K\"othe dual of $\Lambda_{\ph,w}$ is equal to the symmetrization of the K\"othe dual of the weighted Orlicz space $(L_\ph)_w$. Since $(E_w)'=w\cdot (E')_w$ with equal norms, where  $E'$ means the K\"othe dual of $E$, it is natural to introduce for a general symmetric space $E$, not only for K\"othe duals, the ``class'' $M_{E,w}$ defined by
 \[
 M_{E,w}= \{f\in L^0(I): f^*\in w\cdot E_w\}= \{f\in L^0(I): f^*/w\in E_w\}.
 \]
 This class is closed under scalar multiplication but not necessarily by sums, hence it may even not be a linear subspace of $L^0(I)$.  It may be equipped with a gauge $\|f\|_{M_{E,w}} = \|f^*/w\|_{E_w}$, which is positively homogeneous, faithful, monotone and symmetric.  In section \ref{sec:classes-M} we prove that if the weight $w$ is regular then the class $M_{E,w}$ is a linear subspace of $L^0(I)$ and its gauge is equivalent to a norm. The proof of this latter result is based on an optimization formula for the gauge which is of interest by itself, namely
 \[ \|f\|_{M_{E,w}}= \inf\{ \|f/v\|_{E_v}: v\ge 0, v^*=w, \mathrm {supp}\,v\supset \mathrm {supp}\,f\}. \]
 A similar formula was proved in the setting of Orlicz spaces and modulars, in our article \cite{KR}. The proof depended on a certain inequality for rearrangements and weights \cite[Proposition 2.1]{KR}, that cannot have any equivalent form in the present setting. Here this argument is replaced by a completely new one, namely a submajorization formula for rearrangements and weights,  which  is stated and proved in section \ref{sec:an-inequality} (see Theorem \ref{H-L}).
 
 Although the class  $M_{E,w}$ may not be a vector space and its gauge may not be convex, its K\"othe dual space can be defined as the domain of finiteness of the dual function norm $L^0(I)\to [0,+\infty]$,
\[ \|f\|_{(M_{E,w})' } = \sup\left\{ \int_I |fg|\,dm: g\in M_{E,w}, \|g\|_{M_{E,w}}\le 1 \right\}, \]
which is an ideal in $L^0(I)$ normed by the above function norm, and a Banach function space with  Fatou property.   

The next step, performed in section \ref{sec:Koethe duality}, is to show that the K\"othe dual of $M_{E,w}$ coincides isometrically with the Lorentz space $\Lambda_{E',w}$, where $E'$ is the K\"othe dual space of $E$.   The proof of this fact is very similar to that given in the setting of Orlicz-Lorentz spaces in \cite{KR}. As a corollary we obtain that if the weight $w$ is regular, and $E$ has Fatou property, then the K\"othe dual to $\Lambda_{E,w}$ is equal as a set to $M_{E',w}$, with  equivalence of their respective norm and gauge.

In section \ref{sec:Q} we introduce the class $Q_{E,w}$ consisting of all  elements of $L^0(I)$ which are submajorized by some element of $M_{E,w}$. It is easy to verify that $Q_{E,w}$ is an  ideal in $L^0(I)$, which is clearly hereditary by submajorization and contains $M_{E,w}$. The formula
\begin{equation}\label{eq:def-Q-norm} 
\|f\|_{Q_{E,w}}=\inf\{ \|g\|_{M_W}: f\prec g\},
\end{equation}
where the symbol $\prec$ denotes the Hardy-Littlewood submajorization, defines a very natural gauge on $Q_{E,w}$ which turns out to be a norm. Equipped with this norm, $Q_{E,w}$ is a fully symmetric Banach function space, the smallest one  containing $M_{E,w}$. Moreover its K\"othe dual space coincides isometrically with $\Lambda_{E',w}$.

If $E$ has Fatou property one may exchange the roles of $E$ and $E'$, thus $(Q_{E',w})'=\Lambda_{E,w}$, and $(\Lambda_{E,w})'= (Q_{E',w})''$. For deriving our final duality result that $(\Lambda_{E,w})'= Q_{E',w}$, we need to know that $Q_{E',w}$ has Fatou property, and thus equals to its second K\"othe dual. This is shown in section \ref{sec:Q}. A general proof of the latter fact does not seem easy without knowledge of a minimizer  $g$ for the right hand side of the equation (\ref{eq:def-Q-norm}) defining the $Q_{E,w}$ norm of an element $f$. In fact Halperin's level function $f^0$ of the decreasing rearrangement $f^*$ is such a minimizer, in other words we prove that  $\|f\|_{Q_{E,w}}=\|f^0\|_{M_W}$.

At this point we should remark that the path followed here differs from that in \cite{KR}, where the spaces $Q_{\ph,w}$ were not introduced. Instead we initiated there another scale of spaces $P_{\ph,w}$, the analogue of which we define and discuss now.

In section \ref{sec:P} we define the class $P_{E,w}$ consisting of the union of all classes $M_{E,v}$, for all positive decreasing weights $v$ submajorized by $ w$.  This class is equipped with the gauge
\[\|f\|_{P_{E,w}}=\inf\{ \|f\|_{M_{E,v}}: v, 0<v\prec w \}.\]
Contrary to the case of Orlicz-Lorentz spaces, we did not find direct evidence that these classes are linear and these gauges are norms. In the present paper this fact is proven indirectly, at least if $E$ has Fatou property, by showing that $P_{E,w}=Q_{E,w}$, with equality of gauges.

Finally  we obtain three different formulas of the norm in the dual K\"othe space  to  Lorentz space $\Lambda_{E,w}$. In fact we have that for $f\in (\Lambda_{E,w})'$,
 \begin{equation}\label{eq:00} \hskip-10pt
\|f\|_{(\Lambda_{E,w})'} =\inf \{\|g\|_{M_{E',w}}: f\prec g\} = \inf\{\|f\|_{M_{E',v}}: v\prec w, v> 0, v\downarrow\} = \|f^0\|_{M_{E',w}}.
\end{equation}

Let us mention that the expression of the dual norm on  $(\Lambda_{E,w})'$  given in terms of the level function  by equation (\ref{eq:00}) is implicit in Sinnamon's work \cite{S01} (see Theorem 2.2 and Corollary 2.4 there), as it appears clearly once the relationship between Sinnamon's level functions and Halperin's ones has been elucidated like in \cite[p. 64]{FLM}. Our methods however are different and the two infimal expressions in (\ref{eq:00})  seem to be new. 

If $E=L_\varphi$ is an Orlicz space then $\Lambda_{E,w}$ is an Orlicz-Lorentz space $\Lambda_{\varphi,w}$ \cite{KR09, KR}, and we obtain that the norm in its  dual space is expressed in three different ways following from equalities (\ref{eq:00}). But in section \ref{sec:OL}  we consider $L_\ph(J)$ as a modular function space \cite{M}, equipped with its natural convex modular 
\[I_\ph(f)=\int_J \ph(|f|)\,dm.\]
Then the Orlicz-Lorentz space  $\Lambda_{\ph,w}$ inherits of a modular structure defined by the convex modular
$\Phi_w(f)= \int_I\ph(f^*)w\,dm$, while the class $M_{\ph,w}$ is equipped with the (non-convex) modular $\mathrm M_w(f) : =  \int_I \ph\left({f^*\over w}\right)w\,dm$. Set
\[
\mathrm P(f)= \inf\{M_v(f): v\prec w, v> 0, v\downarrow\} \hbox{ and }  \mathrm Q(f)=  \inf\{M_w(g): f\prec g\}.
 \]
These formulas define convex modulars on $L^0$, the associated modular spaces of which coincide with the space $P_{\ph,w}=Q_{\ph,w}$. The modular $\mathrm P$ was introduced in \cite {KR} and further studied in \cite{KLR}, where it was proved that $\mathrm P(f)=\mathrm M_w(f^0)$ under the additional hypothesis that $\ph$ is an $N$-function.  In section \ref{sec:OL} we show that $Q(f)=M_w(f^0)$ (without any hypothesis on $\ph$). Combined with the preceding result of \cite{KLR} we obtain that $Q(f)=P(f)=M_w(f^0)$ if $\ph$ is an $N$-function. 

At the end of section \ref{sec:Lorentz} for generalized Lorentz spaces, as well as in the final section \ref{sec:examples}  for dual spaces, we discuss a number of examples where the space $E$ is more specified. In particular if  $E$ is itself a classical Lorentz space  it turns out that $\Lambda_{E,w}$ is another Lorentz space.


\section{Preliminaries}

 Let $\mu$ be a measure defined on a $\sigma$-algebra $\mathcal{A}$ of subsets of $\Omega$ and $L^0(\Omega,\mathcal{A}, \mu)$ be the set of  all classes of $\mu$-measurable real valued functions  on $\Omega$, modulo equality almost everywhere, and let $L_+^0(\Omega,\mathcal{A}, \mu)$ be the cone of all non-negative functions from $L^0(\Omega,\mathcal{A}, \mu)$. Since  in this article $\Omega$ will typically be an interval of the real line and $\mu$ a measure equivalent to the Lebesgue measure $m$, there will be no ambiguity in the shorter notation $L^0(\Omega)$,  where $\mathcal{A}$ will be implicitly the algebra of Lebesgue-measurable sets. Similarly the space of bounded measurable functions will be unambiguously denoted by $L_\infty(\Omega)$. As for the spaces $L_p(\Omega, \mathcal{A}, \mu)$, $1\le p<\infty$, they depend on $\mu$ and should be denoted by $L_p(\Omega, \mu)$, or simply $L_p(I)$ when $\Omega$ is an interval $I$ and $\mu$ is the Lebesgue measure. When there is no ambiguity, the usual symbols $L_p$ and $L_\infty$ stand for the spaces of $p$-integrable functions and essentially bounded functions on $\Omega$, respectively. Their norms are denoted by $\|f\|_p$ for $f\in L_p$ and $\|f\|_\infty$ for $f\in L_\infty$.

A subset $S$ of $L^0(\Omega)$ is called {\it solid} if for any $f\in L^0(\Omega)$ and $g\in S$ with $|f|\le |g|$ a.e. we have $f\in S$.  An  {\it ideal} in $L^0(\Omega)$ is a solid linear subspace.  {\em A Banach function space} $E$ over $(\Omega, \mathcal{A}, \mu)$, is an ideal in $L^0(\Omega)$, equipped with a monotone  norm $\|\cdot\|_E$, that is  $\|f\|_E \le \|g\|_E$ whenever $|f| \le |g|$ a.e., $f,g\in E$, complete with respect to this norm, and with full support (no element in $L_0$, except 0, is disjoint from all elements in $E$).  The Banach function space $E$ satisfies the {\em Fatou property} whenever for any $f\in L^0(\Omega)$, $f_n \in E$ such that $f_n \uparrow f$ a.e. and $\sup_n \|f_n\|_E < \infty$ it follows that $f\in E$ and $\|f_n\|_E \uparrow \|f\|_E$. We say that $E$ is {\it order continuous} whenever for every sequence $(f_n)\subset E$ with $f_n\downarrow 0$ a.e. we have $\|f_n\|_E \downarrow 0$.
 
 For any $f\in L^0(\Omega)$, we will use the notation $\{f>s\}$ for the set $\{t\in \Omega: f(t) > s\}$ , where the symbol "$>$" can be replaced by $<$, $\le $ or $\ge$.  Throughout the whole paper the terms increasing or decreasing  are reserved for non-decreasing or non-increasing, respectively.
 Given $f\in L^0(\Omega)$, the {\em distribution} of $f$ with respect to $\mu$ is the function $d^\mu_f(s) = \mu\{ |f| > s\}$, $s\ge 0$, and its {\em decreasing rearrangement} $f^{*,\mu}(t) = \inf\{s>0:  d^\mu_f(s)\le t\}$, $t\in [0, \mu(\Omega))$.  Given two measure spaces $(\Omega_i,\mathcal{A}_i,\mu_i)$, $i=1,2$,
 we say that $f_i\in L^0(\Omega_i)$ are {\it equimeasurable}  if $d^{\mu_1}_{f_1}(s) = d^{\mu_2}_{f_2}(s)$, $s\ge 0$, which  equivalently means that   $f_1^{*,\mu_1}= f_2^{*,\mu_2}$. 
  
  A Banach function space  $E$ over $(\Omega, \mathcal{A}, \mu)$ is called a {\em symmetric} space whenever $\|f\|_E= \|g\|_E$ for every equimeasurable functions $f,g\in E$. Recall that the {\em fundamental function} of a symmetric space $E$ is $\phi_E(t) = \|\chi_{A}\|_E$, $\mu(A) =t$, $t\in [0,\mu(\Omega))$. We say that the {\em support} of the symmetric space $E$ is the entire set $\Omega$ whenever $\chi_A \in E$ for any $A \in \mathcal{A}$ with $\mu(A) < \infty$.
  
  The Hardy-Littlewood order  $f\prec_\mu g$ for locally integrable $f,g \in L^0(\Omega)$  is defined   by the inequality $\int_0^x f^{*,\mu} \, dm\le \int_0^x g^{*,\mu}\,dm$ for every $x\in (0,\mu(\Omega))$. If $\Omega = (0,a)$, $a\le \infty$, and $\mu=m$ one writes simply $f\prec g$. Clearly $f\prec_\mu g$ if and only if  $ f^{*,\mu}\prec g^{*,\mu}$.   Recall that $(f+g)^{*,\mu} \prec f^{*,\mu} + g^{*,\mu}$.
  We call   $E$ a {\em fully symmetric} space if $E$ is symmetric and if for any $f\in L^0(\Omega)$ and $g\in E$ with  $f\prec_\mu g$ we have that $f\in E$ and $\|f\|_E\le \|g\|_E$.

The {\em K\"othe dual }space $E'$ of a Banach function space $E$ is the collection of all measurable functions  $f\in L^0(\Omega)$ such that
\[
\|f\|_{E'} = \sup\left\{\int_\Omega |fg|\,d\mu: \|g\|_E \le 1\right\} <\infty.
\]
The space $E'$ equipped with the norm $\|\cdot\|_{E'}$ is a complete Banach function space satisfying the Fatou property. If $E$ is order continuous then the dual space $E^*$ equals the K\"othe dual $E'$, in the sense that the only functionals in $E^*$ are the maps $f\mapsto \int_\Omega fg\, d\mu$, $g\in E'$. If in addition $E$ is a symmetric  space then $E'$ is fully symmetric, and 
 \[
\|f\|_{E'} = \sup\left\{\int_0^{\mu(\Omega)} f^{*,\mu}g^{*,\mu}\,dm: \|g\|_E \le 1\right\}.
\]
For the theory of Banach function and symmetric spaces we refer to the excellent books \cite{BS, KPS, Z}.  

 Given $f,g\in L^0(\Omega)$ denote $f\wedge g = \min\{f, g\}$ a.e., $f\vee g = \max\{f,g\}$ a.e., $f_+ = f\vee 0$ a.e. and $f_- = -f\vee 0$ a.e..  By $m$ denote always the Lebesgue measure on  subsets of real numbers $\mathbb{R}$. Recall that for $f\in L_1+L_\infty(\Omega)$, $ x\in (0,\mu(\Omega))$,
  \begin{align}\label{eq:1}
\int_0^x f^{*,\mu}\,dm &=\inf\{\|g\|_1+x\|h\|_\infty: g\in L_1, h\in L_\infty, f=g+h\} \\
& = \inf_{\lambda>0}\, \left[ \int_\Omega (|f|-\lambda)_+\, d\mu+ \lambda x\right]\notag
\end{align}
(see e.g. Theorem 6.2 in  \cite[Ch. 2]{BS} and its proof; Exercise 1 on p. 87). It is well known (cf. Proposition 1.8, p.43, \cite{BS}) that for any $0<p<\infty$,
\[
\int_\Omega |f|^p \, d\mu = \int_0^{\mu(\Omega)} (f^{*,\mu})^p\, dm,
\]
in which formula we can replace  $|f|^p$ by $\varphi(|f|)$ where $\varphi:\mathbb{R_+} \to \mathbb{R_+}$ is any increasing continuous function.

 Let $I=(0,a)$, where $0<a\le \infty$, and $L^0 = L^0(I)$ be  the space of all real valued Lebesgue measurable functions on $I$.
If $\Omega = I$ and $\mu = m$ then the distribution and decreasing rearrangement of a measurable function $f$ are denoted by $d_f$ and $f^*$, respectively. The support of $f$ is denoted by $\mathrm{supp}\,f$.

Let us recall a useful connection between a measurable function and its decreasing rearrangement. Let $f$ be a measurable function on $I$ and $f^*$ be its decreasing rearrangement.
\begin{proposition}\label{prop:ryff}\hfill
\begin{itemize}
\item[(i)] \cite[Ryff's  Theorem 7.5]{BS}  If $a<\infty$, or if $\mathrm{supp}\,f$ has finite measure, there exists an onto and measure preserving transformation $\tau: I\to I$, that is $\tau$ is measurable and $m(\tau^{-1}(A))=m(A)$ for each measurable subset of $I$, such that $|f| = f^*\circ \tau$.

\item [(ii)]  \cite[Corollary 7.6]{BS}  If $\mathrm{supp}\,f$ has infinite measure, and $\lim_{t\to \infty} f^*(t)=0$, then such a measure preserving transformation $\tau$ exists but only from $\mathrm{supp}\,f$ onto the support of $f^*$. The equation  $|f|(t) = f^*\circ \tau(t)$ is valid for  $t\in \mathrm{supp}\,f$. 
\end{itemize}
\end{proposition}

We shall need to consider a third case that we settle as follows.

\begin{lemma}\label{lem:Ryff-3}
Let $I=(0,\infty)$ and $f$ be a measurable function in $I$ such that $\lim_{t\to \infty} f^*(t)=\alpha>0$. Then for each $\eps>0$ there exists an onto and measure preserving transformation $\tau: I\to I$ such that $|f|\le (1+\eps)\, f^*\circ\tau$. 
\end{lemma}
\begin{proof}
Set $\tilde f= |f|\vee (1+\eps)\alpha$. Note that $(\tilde f)^* = f^*\vee (1+\eps)\alpha\le (1+\eps)\,f^*$. Since $m\{|f|\ge (1+\eps)\alpha\} < \infty$, by Ryff's theorem  we may find an onto measure preserving transformation $\tau: I\to I$ such that
$(|f|-(1+\eps)\alpha)_+=(f^*-(1+\eps)\alpha)_+\circ \tau$. Adding the constant $(1+\eps)\alpha$ to both sides yields  $\tilde f=(\tilde f)^*\circ \tau$. Then\medskip

{}\hfill$ |f|\le \tilde f= (\tilde f)^*\circ\tau \le (1+\eps)\,f^*\circ\tau$.
\end{proof}

 We will assume throughout the paper that $w: I\to (0,\infty)$ is a decreasing positive {\em weight} function.
    Then $d\,\omega=wdm$ is a measure on $I$ such that $\omega(A) = \int_A w\, dm$ for Lebesgue measurable subsets
  $A\subset I$.   The symbols $d_f^w$ and $f^{*,w}$ will be reserved for  the distribution and decreasing rearrangement of $f$ respectively,  with respect to the measure $\omega$.  Define
   \[
  W(t) = \int_0^t w\, dm, \ \ \ \  t\in I, \ \ \ W(\infty) =\int_0^\infty w \, dm\ \ \ \ \text{if} \ \ \ \ I=(0,\infty). 
\]    
  Let further  $b=\omega(I) =W(a) \in (0,+\infty]$  and $J = (0,b)$. The interval $J$ will be always equipped with the Lebesgue measure $m$. It may happen that $a=\infty$ and $b<\infty$ if $w$   is integrable on $I$, or that $a<\infty$ and $b=\infty$ if $w$ is not integrable near $0$. If the weight $w$ is integrable near $0$, it is integrable on any finite interval,  and then
clearly $W(t) < \infty$ for all $t\in I$.   We say that the weight $w$ is {\em regular} if $W(t)\le C tw(t)$ for some $C\ge 1$ and all $ t\in I$.

Throughout the paper   the symbol $E$ will always stand for a fully symmetric Banach function space contained in $L^0(J)$ with its support equal to $J$.

 \section{Lorentz spaces $\Lambda_{E,w}$}\label{sec:Lorentz}
 
 \subsection{Spaces $E_w$}\label{subsec:21}

Given a fully symmetric space $E\subset L^0(J)$, let $E_w$ be the subset of $L^0=L^0(I)$ and $\|\cdot\|_{E_w}$ the functional on $E_w$ such that
\[
E_w = \{ f\in L^0: f^{*,w}\in E\}, \ \ \ \|f\|_{E_w} = \|f^{*,w}\|_E, \ \  f\in E_w.
\]
The space $E_w$ is a fully symmetric space on $I$ for the measure $\omega$. 
Note that if $f\in L^0(I)$ then $f^{*,w} \in L^0(J)$.  If $E= L_p(J)$, $1\le p < \infty$, then $E_w = (L_p)_w$ is traditionally called a weighted $L_p$ space on $I$, which is not symmetric with respect to the measure $m$. However this is  an ordinary $L_p$-space on $(I,\omega)$ in the sense that  for  $f\in E_w = (L_p)_w$ we have \cite[Proposition 1.8, p.43]{BS}
\[ 
\int_J (f^{*,w})^p dm= \int_J (|f|^p)^{*,w} dm = \int_I |f|^p d\omega = \int_I |f|^p wdm,
\]
so that  $\|f\|_{(L_p)_w} = (\int_I |f|^p w dm)^{1/p}$. Clearly it is symmetric with respect to the measure $\omega$. 

Let $\varphi:[0,\infty)\to [0,\infty)$ be an Orlicz function, that is $\varphi(0)=0$, $\varphi$ is convex and positive on $(0,\infty)$. Then for $f\in L^0(J)$ define the Orlicz modular as 
\[
I_\varphi(f)= \int_J \varphi(|f|)\, dm,
\]
and the Orlicz space $L_\varphi(J)$ \cite{BS} as a collection of $f\in L^0(J)$ such that for some $\lambda > 0$, 
$I_\varphi(|f|/\lambda)<\infty.$
 It is a Banach fully symmetric space equipped with either of the norms, the Luxemburg  norm $\|f\|_\varphi = \inf\{\lambda > 0: I_\varphi(|f|/\lambda) \le 1\}$ or the Orlicz norm $\|f\|^0_\varphi = \inf_{t>0}t(1 + I_\varphi(f/t))$.  Analogously as for $L_p$-spaces, if  $E = L_\varphi(J)$ then $E_w = (L_{\varphi})_w$
 is a weighted Orlicz space symmetric with respect to the measure  $\omega$, associated with the Orlicz modular
 \[ 
\int_J \varphi(f^{*,w})\,dm=\int_I \varphi(|f|)\, w dm.
 \]

\begin{remark}\label{rem:21}
The space $E_w$ over $(I, \omega)$ where $d\omega = wdm$ can be called a generalized weighted space induced by the space $E$ over $(J,m)$ and the weight $w$ on $I$.  {\it In general, $E_w$ is a  Banach function space in $L^0(I)$ which is non symmetric with respect to the Lebesgue measure but isometrically order isomorphic to $E$ on $(J,m)$. } 

This is a simple consequence of a general theorem of Caratheodory on isomorphisms of separable atomless measure algebras \cite[Chap. 15, Theorem 4]{Ro}, but a far more elementary proof may be given in the present case.

Indeed, there exists a bijective, bimeasurable map $S:I\to J$ which is measure preserving i.e. $m(S(A))=\omega(A)$ for all measurable $A\subset I$. This result follows from general facts in measure theory, but such a map will be explicitly exhibited below. Since for every $f\in L^0(I)$ and $t>0$ we have $\{|f\circ S^{-1}|>t\}=S\{|f|>t\}$ we see that 
\begin{align}\label{eq:distrib}
d_{f\circ S^{-1}}(t)= m(|f\circ S^{-1}|>t)=m(S\{|f| > t\}) = \omega(|f|>t)=d^w_f(t).
\end{align} 
Hence $(f\circ S^{-1})^*= f^{*,\omega}$. Thus $f\in E_w$ if and only if $f\circ S^{-1}\in E$ and $\|f\|_{E_w}= \|f^{*w}\|_E =\|(f\circ S^{-1})^*\|_E =\|f\circ S^{-1}\|_E$. 
The map  $T:L^0(I)\to L^0(J):f\mapsto f\circ S^{-1}$  is a linear order isomorphism, so that  $E_w=T^{-1}(E)$ must be an ideal of $L^0(I)$. The restriction of  $T$ to $E_w$ is the wished Banach lattice isometry.

Now for the sake of constructing a map $S$ as requested in the preceding paragraph, we consider two cases.

a) If $W<\infty$ on $I$, then $W$ is a bijective, bimeasurable, measure preserving  map from $(I,\omega)$ onto $(J,m)$, so that we may set $S=W$. 

Indeed, since $w>0$ is integrable on every finite segment $(0,x)\subset I$, the map $W$ is a homeomorphism from $I$ onto $W(I)=J$. The pushforward measure of $\omega$ by $W$ is $\omega\circ W^{-1}=m$, the Lebesgue measure, as can be seen easily on intervals $[x,y]\subset I$,
\[
\omega(W^{-1}([x,y]))=\omega([W^{-1}(x),W^{-1}(y)])=\int_{W^{-1}(x)}^{W^{-1}(y)} w\,dm= y-x=m([x,y]),
 \]
it follows that $m(W(A))=\omega(A)$ for all measurable $A\subset I$.

b) If $W(t)=\infty$ for $t>0$ we choose $\alpha\in I=(0,a)$, and  set $W_\alpha(t)= \int_\alpha^t w\, dm$ for $t\in I$. Letting $c=\int_\alpha ^a w\, dm$, and $K=(-\infty,c)$,  $W_\alpha$ is a  bijective, bimeasurable, measure preserving  map from $(I,\omega)$ onto $(K, m)$. It is then a standard exercise to exhibit a bijective, bimeasurable, measure preserving map $U$ from $(K, m)$ onto $(J,m)$, and we set $S=U\circ W_\alpha$.
\end{remark}

Since the case $W<\infty$ is  the main one considered in this article, except in sections \ref{sec:an-inequality} and \ref{sec:classes-M},  we collect the preceding information relative to this case in the following proposition.

\begin{proposition} \label{prop:S=W}
Assume that $W<\infty$ on $I$. Then

\rm{(i)} Every $f\in L^0(I)$ is equimeasurable with respect to $\omega$ to $f\circ W^{-1}\in L^0(J)$ with respect to $m$. Consequently, 
\[
(f\circ W^{-1})^*= f^{*,w}.
\] 

\rm{(ii)} $f\in L^0(I)$ belongs to $E_w$ if and only if $f\circ W^{-1}$ belongs to $E$, and then 
\[
\|f\|_{E_w}=\|f\circ W^{-1}\|_E.
\]
 Consequently, $E_w$ is an  ideal in $L^0(I)$, it is  fully symmetric for the measure $d\omega = wdm$, and the map $f\mapsto f\circ W^{-1}$ induces an order isometry from $E_w$ onto $E$.
\end{proposition}

\subsection{Generalized Lorentz spaces}
Define now the {\em Lorentz space} $\Lambda_{E,w}$  as the symmetrization of $E_w$ \cite{KR}, that is 
\[
\Lambda_{E,w} = \{f \in L^0(I):  f^*\in E_w\}, \ \ \ \ \|f\|_{\Lambda_{E,w}} = \|f^*\|_{E_w}.
\]

If $W(t)=\infty$ for $t>0$, then $J= (0,\infty)$ and if $f$ is a decreasing nonnegative function in $L^0(I)$, then  $d_f^w=\infty\cdot\chi_{[0,f(0_+)]}$ and  $f^{*,w}= f(0_+) \cdot\chi_J$.  It follows  that $\Lambda_{E,w}=\{0\}$ except if $E$ contains the function $1$, in which case $\Lambda_{E,w}=L_\infty(I)$.

For the rest of this section we disregard the above degenerate case and assume that $W <\infty$ on $I$.

For the Orlicz space $E=L_\ph(J)$,  $\Lambda_{E,w}$   is  the  Orlicz-Lorentz space $\Lambda_{\ph,w}$, defined in \cite{KR}, that is  $\|f\|_{\Lambda_{\varphi,_w}} = \|f^*\|_{{(L_\varphi)}_w}$. If $\varphi(t)= {t^p}$, $1\le p < \infty$, then  $\Lambda_{E,w} = \Lambda_{p,w}$ \cite{CRS, Ha53}. 

If $E=L_\infty(J)$ then $E_w= L_\infty(I)= \Lambda_{E,w}$.

Other examples are given at the end of the present section.

\begin{proposition}\label{prop:2}
Let $W<\infty$ on $I$.
\begin{itemize}
\item[\rm(i)] The support of $\Lambda_{E,w}$ is $I$.
\item[\rm(ii)] For all $f \in \Lambda_{E,w}$,
 \[
  \|f\|_{\Lambda_{E,w}} = \|f^*\circ W^{-1}\|_E.
 \] 
 \item[\rm(iii)] The functional $\|\cdot\|_{\Lambda_{E,w}}$ is a norm, and the Lorentz space $\Lambda_{E,w}$ is a fully symmetric Banach space. If $E$ has  Fatou property   then $\Lambda_{E,w}$ has also this property.  If $E$ is order continuous then $\Lambda_{E,w}$ is also order continuous.
\end{itemize}
\end{proposition}

\begin{proof} $\rm (i)$ Let $A\subset I$ with $m(A) < \infty$. Then $W(m(A))< \infty$ and
\[
\|\chi_A\|_{\Lambda_{E,w}} = \|\chi_{(0,m(A))}\|_{E_w} =\|\chi^{*,w}_{(0,m(A))}\|_E = \|\chi_{(0,W(m(A)))}\|_E < \infty
\]
 since by assumption the support of $E$ is $J$.

$\rm(ii)$
 In view of $w > 0 $ on $I$, the function $W:I\to J$ is a strictly  increasing homeomorphism.  By Proposition \ref{prop:S=W}  the functions $f$ for $\omega$ and $f\circ W^{-1}$ for $m$  are equimeasurable, that is $d^w_f = d_{f\circ W^{-1}}$. 
 So $f^{*,w}=(f\circ W^{-1})^*$ and hence 
 \[
\|f\|_{\Lambda_{E,w}} =  \|f^*\|_{E_w} = \|(f^*)^{*,w}\|_E = \|f^*\circ W^{-1}\|_E.
 \]
$\rm(iii)$
For $f\in L_1+L_\infty$ and $g\in \Lambda_{E,w}$ with $f\prec g$, and $x\in J$ we have
\[
\int_0^x f^*\circ W^{-1}\,dm=\int_0^{W^{-1}(x)} f^*w\,dm\le \int_0^{W^{-1}(x)} g^* w\,dm = \int_0^x g^*\circ W^{-1}\,dm
\]
by  Hardy's inequality \cite[Proposition 3.6, Ch.2]{BS}. Hence $f^*\circ W^{-1}\prec g^*\circ W^{-1}\in E$ and so by the assumption of full symmetry of $E$  and by (ii) we get $f^*\circ W^{-1}\in E$, hence $f\in \Lambda_{E,w}$, and
\[
\|f\|_{\Lambda_{E,w}}=\|f^*\circ W^{-1}\|_E\le \|g^*\circ W^{-1}\|_E=\|g\|_{\Lambda_{E,w}}.
\]
Now if $f,g\in \Lambda_{E,w}$ we have $f^*,g^*\in E_w$, hence $f^*+g^*\in E_w$, which means that $f^*+g^*\in \Lambda_{E_w}$. Moreover $\|f^*+g^*\|_{\Lambda_{E_w}}= \|f^*+g^*\|_{E_w}\le \|f^*\|_{E_w}+\|g^*\|_{E_w}=\|f^*\|_{\Lambda_{E_w}}+\|g^*\|_{\Lambda_{E_w}}$. Then by the well known submajorization  $(f+g)^*\prec f^*+g^*$ \cite[Theorem 3.4]{BS}, it follows from the preceding observation that $f+g\in \Lambda_{E,w}$ and 
\[\|f+g\|_{\Lambda_{E,w}}\le \|f^*+g^*\|_{\Lambda_{E,w}}\le \|f^*\|_{\Lambda_{E,w}}+\|g^*\|_{\Lambda_{E,w}}=\|f\|_{\Lambda_{E,w}}+\|g\|_{\Lambda_{E,w}}\]
 Therefore $\|\cdot\|_{\Lambda_{E,w}}$ is a fully symmetric norm.

The  normed function space $\Lambda_{E,w}$ is complete since it is a symmetrization of the complete space $E_w$ \cite[Lemma 1.4]{KR}.

Suppose now that $E$ has the Fatou property. Take $f_n, f\in L^0(I)$, $f_n\uparrow f$ a.e., and $\sup\|f_n\|_{\Lambda_{E,w}}<\infty$.  Then $f^*_n \circ W^{-1} \uparrow f^* \circ W^{-1}$ a.e.,  and by (ii) $\sup\|f^*_n\circ W^{-1}\|_E = \sup  \|f_n\|_{\Lambda_{E,w}} < \infty$. Now by the Fatou property of $E$, $ f^*\circ W^{-1} \in E$ so $f\in \Lambda_{E,w}$, and  $\|f_n\|_{\Lambda_{E,w}} = \|f^*_n \circ W^{-1}\|_E \uparrow \|f^*\circ W^{-1} \|_E =\|f\|_{\Lambda_{E,w}}$.  The statement on order continuity of $\Lambda_{E,w}$ can be proved analogously.
\end{proof}

\subsection*{Applications.} Proposition \ref{prop:2}(ii) allows to compute some Lorentz spaces.

\begin{example}[Reiteration]\label{ex:reiter} {\it Let $w_1,w_2$ be two  locally integrable decreasing positive weights on $I_1=(0,a_1)$, resp. $I_2 =(0,W_1(a_1))$, where $W_1(x) =\int_0^x w_1\, dm$ for $x\in I_1$,  and $W_2(x)=\int_0^x w_2\, dm$ for $x\in I_2$. For every symmetric space $E$ on $J=(0,b)$, $b=W_2(W_1(a_1))$, it holds that $\Lambda_{\Lambda_{E,w_2},w_1} =\Lambda_{E,w}$ with equal norms, where $w=(w_2\circ W_1)w_1$}.
\end{example}
\begin{proof}
For  $f\in L^0(I_1)$ we have $f\in \Lambda_{\Lambda_{E,w_2},w_1}$ if and only if $f^*\circ W_1^{-1}\in \Lambda_{E,w_2}(I_2)$, that is if $(f^*\circ W_1^{-1})\circ W_2^{-1}$ belongs to $E$.  Setting $W=W_2\circ W_1$, $W$ is an increasing concave function with a derivative defined almost everywhere by $W'= (w_2\circ W_1)w_1 =: w$,  which is a decreasing weight on $I_1$. Then  $f\in \Lambda_{\Lambda_{E,w_2},w_1}$ if and only if $f^*\circ W^{-1}\in E$, that is $f\in \Lambda_{E,w}$. The fact that both norms coincide is straightforward.
\end{proof}

For definition of the Marcinkiewicz space $M_W$ see section \ref{sec:Q}.

\begin{example}[Marcinkiewicz-Lorentz spaces] {\it Let $I_1,I_2,w_1,w_2$ be as in Example \ref{ex:reiter} and $M_{W_2}(I_2)$ be the  Marcinkiewicz space associated with the weight $w_2$. Then the Lorentz space $\Lambda_{M_{W_2},w_1}$ consists of   $f\in L^0(I_1)$ such that }
\[
\|f\|:=\sup_{x\in I_1}\frac1{W_2\circ W_1(x)}\int_0^xf^* w_1\,dm <\infty.
\]
\end{example}
\begin{proof}
Clearly $f\in \Lambda_{M_{W_2},w_1}$ if and only if  $f^*\circ W_1^{-1}\in M_{W_2}(I_2)$, that is 
\[ \|f^*\circ W_1^{-1}\|_{M_{W_2}}=\sup_{t\in I_2} \frac 1{W_2(t)}\int_0^t f^*\circ W_1^{-1}(s)\,ds <\infty.
\]
The result follows by performing first  the substitution for $W^{-1}_1(s)$ in the integral, then the change $t=W_1(x)$ in the supremum.
\end{proof}

Recall  if $(E, \|\cdot\|_E)$ and $(F, \|\cdot\|_F)$ are two fully symmetric Banach function spaces over the same interval $J$, then the Banach function spaces $E\cap F$ and $E+F$ equipped with the standard norms $\|f\|_{E\cap F} = \max\{\|f\|_E, \|f\|_F\}$ and $\|f\|_{E+ F} = \inf\{\|f_1\|_E + \|f_2\|_F: f=f_1 + f_2, f_1\in E, f_2\in F\}$ respectively, are also fully symmetric. This is evident for the intersection space $E\cap F$, while for the sum space $E+F$ it is an immediate consequence of the following decomposition property for the submajorization.

\begin{fact}\label{fact:decomp}
 If $f, g_1,g_2\in L^0_+$  are locally integrable with $f\prec g_1+g_2$ then there is a decomposition $f=f_1+f_2$ into  non-negative functions such that $f_1\prec g_1$ and  $f_2\prec g_2$.
\end{fact}

This fact is an easy consequence of the well known characterization of submajorization by Calder\'on, namely that $f\prec g$ if and only if there exists a substochastic linear operator $T$ such that $|f|=T|g|$ (\cite[Theorem II-3.4]{KPS}, or \cite[Chap.3, Proposition 2.4 and Theorem 2.10]{BS}).

In the following example we shall use a monotone version of Fact \ref{fact:decomp}, that is based on a monotone refinement of Calder\'on's theorem by Bennett and Sharpley \cite[Theorem 5]{BS2}, \cite[Remark 7.6, Theorem 7.7]{BS} (see also \cite[\S 3]{L} for a different proof), i.e. if $f,g$ are non-negative  locally integrable and decreasing functions, such that $f\prec g$ then $f=Tg$ for some positive substochastic operator $T$ {\it which preserves the cone of decreasing  non-negative functions}. Thus we obtain. 
\begin{fact}\label{fact:decomp2}
 {\it If $f,g_1,g_2$ are non-negative  decreasing locally integrable functions with $f\prec g_1+g_2$ then there is a decomposition $f=f_1+f_2$ into  non-negative  decreasing functions such that $f_1\prec g_1$, $f_2\prec g_2$.}
\end{fact}

\begin{example}[Intersections and sums] {\it Let $E, F$ be fully symmetric Banach function spaces defined on the same interval $J$, and $w$ a locally integrable decreasing positive weight on $I$ with $W(I)=J$. Then $\Lambda_{{E\cap F},w} = \Lambda_{E,w}\cap \Lambda_{F,w}$ and $\Lambda_{{E+ F},w} = \Lambda_{E,w}+ \Lambda_{F,w}$ with equality of norms.}
 \end{example}

\begin{proof}
 The formula for the Lorentz space of an intersection is straightforward, so we treat only the sum case. 
 
From $E\subset E+F$, with norm-decreasing inclusion it follows immediately that $\Lambda_{E,w}\subset \Lambda_{{E+ F},w}$, with norm-decreasing inclusion. Similarly $\Lambda_{F,w}\subset \Lambda_{{E+ F},w}$, and thus $\Lambda_{E,w}+ \Lambda_{F,w}\subset \Lambda_{{E+ F},w}$. Moreover  this inclusion is norm-decreasing.

As for the converse inclusion, let $f\in\Lambda_{E+ F,w}$. We have $f^*\circ W^{-1}\in E+F$, hence for any $\eps>0$ there are $g\in E, h\in F$ such that $f^*\circ W^{-1}=g+h$ and $\|g\|_E+\|h\|_F\le (1+\eps) \|f^*\circ W^{-1}\|_{E+F}$. Then $f^*\circ W^{-1}\prec g^*+h^*$, and by  of Fact \ref{fact:decomp2} there exist decreasing  non-negative   functions $g_1,h_1$ such that 
 \[g_1\prec g^*, h_1\prec  h^*\hbox{ and } f^*\circ W^{-1}= g_1+h_1.\]
We have then $g_1\in E$ and $h_1\in F$. Setting $k=g_1\circ W$, $l=h_1\circ W$, we have $f^* = k+l$.   Since $k,l$ are  non-negative decreasing and $k\circ W^{-1}\in E$, $l\circ W^{-1}\in F$, we have $k\in \Lambda_{E,w}$, $l\in \Lambda_{F,w}$ with $\|k\|_{\Lambda_{E,w}}=\|g_1\|_E\le \|g\|_E$, $\|l\|_{\Lambda_{F,w}}=\|h_1\|_F\le \|h\|_F$.  It follows $f\in \Lambda_{E,w}+\Lambda_{F,w}$ with 
\[\|f\|_{\Lambda_{E,w}+\Lambda_{F,w}}\le \|g\|_E+\|h\|_F\le (1+\eps)\|f^*\circ W^{-1}\|_{E+F} =(1+\eps) \|f\|_{\Lambda_{E+F, w}}\]
\end{proof}

\section{An inequality for rearrangements of functions and weights}\label{sec:an-inequality}

Let $v\in L_+^0= L_+^0(I)$, $I=(0,a)$. It defines a measure $d\nu=vdm$ on $I$ in the usual way by setting $\nu(A) = \int_A v\, dm$, where  $A\subset I$ is Lebesgue measurable. If $f\in L^0$ then by  $f^{*,v}$ we denote the decreasing rearrangement of $f$ with respect to the measure $\nu$. This is a decreasing function on the interval $J_v:=(0,\nu(I))$. Clearly 
$f=\chi_{\{v>0\}}\,f$ $\nu$-a.e., so $f^{*,v}=(\chi_{\{v>0\}} f)^{*,v}$.  If $v$ has a rearrangement $v^*$ such that  $v^* = w$, then we have 
 \begin{equation}\label{eq:31}
 \nu(I)=\int_I v = \int_I v^*= \int_I w = \omega(I)=b,
 \end{equation}
  and so $J_v=(0,b) = J$ does not depend on $v$ in that case. If $E$ is a symmetric space on  $J$ then $E_v$ is defined as in the case of a decreasing weight by $f\in E_v\iff f^{*,v}\in E$, where $f^{*,v}$ is the decreasing rearrangement of $|f|$ relative to the measure $\nu$. Then again, $E_v$ is a symmetric Banach function space on   $I$ equipped with the measure $\nu$, which is order-isometric to $E$.

If $\mathrm{supp}\, f \subset \mathrm{supp}\, v$ then we agree that $(f/v)(t) = 0$ for $t\notin \mathrm{supp}\, f$. 

\begin{theorem}\label{H-L}
Let $v\in L_+^0$ be such that $v^* = w$.  Assume $f\in L_1+ L_\infty (I)$ with $\mathrm{supp}\,f\subset \mathrm{supp}\,v$. Then
\[
 \left(\frac{f^*}w\right)^{*,w} \prec \left(\frac fv\right)^{*,v}.
 \]
In particular if $f/v\in E_v$ then $f^*/w\in E_w$ and $\|f^*/w\|_{E_w}\le \|f/v\|_{E_v}$.
\end{theorem}

We prove first two lemmas.

\begin{lemma}\label{rearrangement-of-inf}
For any $f,g\in L_+^0$ we have $(f\wedge g)^*\le f^*\wedge g^*$.
\end{lemma}
\begin{proof}
First notice that $m(\{f^*>s\}\cap\{g^*>s\})= m\{f^* > s\}\wedge  m\{ g^* > s\} $, $s\ge 0$,  since the sets $\{f^*>s\}$ and $\{g^*>s\}$ are two intervals with the same lower bound $0$.
Thus we have 
\begin{align*}
d_{f\wedge g} (s) &= m\{f\wedge g >s \} = m(\{f>s\}\cap \{g>s\})\\
&\le m\{f>s\}\wedge m\{g>s\}
=m\{f^*>s\}\wedge m\{g^*>s\}\\
&=m(\{f^*>s\}\cap\{g^*>s\})=m\{f^*\wedge g^* >s\} = d_{f^*\wedge g^*}(s),
\end{align*}
which implies $(f\wedge g)^* \le f^* \wedge g^*$. 
\end{proof}
\begin{lemma}\label{comparison}
For every  $f,g\in L_+^0$ such that $f^*, g^* < \infty$, it holds
\[
\int_I(f^*-g^*)_+dm\le \int_I (f-g)_+dm.
\]
\end{lemma}

\begin{proof}
We assume first that $0\le f$ is bounded. Note that
\[
(f-g)_+= f- f\wedge g \qquad (f^*-g^*)_+= f^*- f^*\wedge g^*.
\]
Then by Lemma \ref{rearrangement-of-inf}  we have for every $t\in I $,
\[
\int_0^t (f^*-g^*)_+dm=\int_0^t(f^*- f^*\wedge g^*)dm\le\int_0^t (f^*-(f\wedge g)^*))dm.
\]
But since $f^*\prec (f-f\wedge g)^*+(f\wedge g)^*$ and  $\int_0^t (f\wedge g)^*<\infty$ by boundedness of $f$,
\[
\int_0^t (f^*-(f\wedge g)^*))dm= \int_0^t f^*dm-\int_0^t(f\wedge g)^*dm\le \int_0^t(f-f\wedge g)^*dm=\int_0^t[(f-g)_+]^*dm.
\]
Therefore  for every $t\in I= (0,a)$,
\[
\int_0^t (f^* - g^*)_+ \le \int_0^t [(f - g)_+]^*.
\]
Letting $t\uparrow a$ we obtain
\[
\int_0^a (f^*-g^*)_+\,dm\le \int_0^a[(f-g)_+]^*dm=\int_0^a (f-g)_+\,dm.
\]
If $0\le f$ is not bounded, letting $f_n=f\wedge n$, $n\in \mathbb{N}$, we get $f_n^*\uparrow f^*$ a.e. and thus $(f_n^* - g^*)_+ \uparrow (f^* - g^*)_+$ a.e. as well as $(f_n - g)_+\uparrow (f-g)_+$ a.e.. Now by the monotone convergence theorem, 
\[
\int_I (f^* - g^*)_+ \, dm = \lim_{n\to\infty}\int_0^a (f_n^*-g^*)_+\,dm\le \lim_{n\to\infty}\int_0^a (f_n-g)_+\,dm = \int_I (f-g)_+\, dm.
\]
\end{proof}

\begin{remark}
Using Lemma \ref{rearrangement-of-inf} and Lorentz-Shimogaki inequality \cite[Chapter 3, Theorem 7.4]{BS} for rearrangements, we obtain in fact the more powerful result
\[
(f^*-g^*)_+\prec (f-g)_+.
\]
Indeed since  $f\ge f\wedge g$, Lorentz-Shimogaki's theorem gives $f^* - (f\wedge g)^* \prec f - f\wedge g$ and 
\[
(f^*-g^*)_+=f^*-f^*\wedge g^*\le f^*-(f\wedge g)^*\prec f - f\wedge g =  (f-g)_+.
\]
However Lemma \ref{comparison}, which requires only quite elementary ingredients in its proof, will suffice for our purpose. 
\end{remark}

\begin{proof}[Proof of Theorem \ref{H-L}]
By Lemma  \ref{comparison}, for every $\lambda>0$ we have
\begin{align*}
\int_I \left(\frac{f^*}w-\lambda\right)_+ w\, dm
&= \int_I \left({f^*}-\lambda w\right)_+ \,dm
= \int_I (f^* - (\lambda v)^*)_+\, dm\\
&\le \int_I (|f| - \lambda v)_+ \, dm 
= \int_I \left(\frac{|f|}{v} - \lambda\right)_+ v \, dm.
\end{align*}
Now in view of the equality (\ref{eq:1}), for any $x\in J$,
\begin{align*}
\int_0^x \left(\frac{f^*}{w}\right)^{*,w} \,dm &= \inf_{\lambda>0}\left[\int_I\left(\frac{f^*}{w} - \lambda\right)_+ wdm + \lambda x \right]\\
&\le \inf_{\lambda>0}\left[\int_I\left(\frac{|f|}{v} - \lambda\right)_+ v\,dm + \lambda x \right]
= \int_0^x \left(\frac{f}v\right)^{*,v} \,dm,
\end{align*}
and the proof is completed.
\end{proof}

\begin{proposition}\label{lem:equivalence}
Let  $f \in L^0$ have a finite decreasing rearrangement $f^*$. If $I$ is a finite interval $(0,a)$, or $I= (0, \infty)$ with $\lim_{t\to\infty}f^*(t)= 0$, then there exists  $v\in L_+^0$ such that 
\[
v^*=w, \ \ \ \mathrm{supp}\, v\supset \mathrm{supp}\, f \ \ \text{and}\ \ \  \left(\frac{f^*}{w}\right)^{*,w} = \left(\frac{f}{v}\right)^{*,v}.
\]
 If $I=(0,\infty)$ and $\lim\limits_{t\to\infty}f^*(t)>0$ then for every $\eps>0$ there exists $0<v\in L^0$ such that 
 \[
v^* = w \ \ \ \text{and}\ \ \ \ \left(\frac{f}{v}\right)^{*,v}\le (1+\eps) \left(\frac{f^*}{w}\right)^{*,w}.
 \]
\end{proposition}

\begin{proof}
 The proof will make use of the following fact. \\
 \indent  (a) If $\tau: I\to I$ is a measure preserving transformation, $w$ is a weight on $I$ and $v=w\circ \tau$, then clearly $v^*=w$. Moreover  for every $h\in L^0_+$ we have  $h^{*,w} = (h\circ \tau)^{*,v}$. 
 
 Indeed  for every $\lambda>0$, and $g\in L^0_+$ we have $ \{g\circ\tau>\lambda\}=(g\circ \tau)^{-1} (\lambda, \infty) = (\tau^{-1}\circ g^{-1})(\lambda, \infty) = \tau^{-1}(\{g>\lambda\})$. Thus $g\circ\tau$ and $g$ are equimeasurable for the measure $m$, and it follows that  $\int_I g\circ\tau\,dm=\int_I g\,dm$.  Setting now $g=\chi_{\{h>\lambda\}}w$, we get
 \begin{align*}
\omega(\{h>\lambda\}) &= \int_I \chi_{\{h>\lambda\}}\,wdm = \int _I(\chi_{\{h>\lambda\}}\circ\tau)\,(w\circ\tau)\,dm \\
&= \int_I \chi_{\{h\circ \tau >\lambda\}}v\,dm =\nu(\{h\circ \tau >\lambda\}),
\end{align*}
hence $h$ for $\omega$ and $h\circ\tau$ for $\nu$ are equimeasurable, and so $h^{*,w} = (h\circ \tau)^{*,v}$.

Let us  come back now to the proof of Proposition \ref{lem:equivalence} itself. We consider several cases.

 If the support of $f$ has finite measure, then by Proposition \ref{prop:ryff} (i), there exists a measure preserving onto transformation $\tau: I\to I$ such that $|f|(t) = f^*\circ \tau(t)$, $t\in I$. Then setting $v = w\circ \tau$, we have $v>0$, $v^*=w$, and by (a), $f^*/w$ for $d\omega=wdm$ and $(f^*/w)\circ\tau$ for $d\nu=vdm$ are equimeasurable. But $(f^*/w)\circ\tau = (f^*\circ\tau)/(w\circ\tau) = |f|/v$,  and the desired equality of rearrangements follows.

If now the support of $f$ has infinite measure and $\lim_{t\to \infty} f^*(t) =0$, by Proposition \ref{prop:ryff} (ii) there exists a measure preserving transformation $\tau$ from the support of $f$ onto the support of $f^*$, such that  $|f|(t) = f^*\circ \tau(t)$ for $t\in\mathrm{supp}\,f$. Define $v(t) = w \circ \tau(t)$ for $t$ in the support of $f$ and $v(t)=0$ otherwise. By the assuption that the support $f$ has infinite measure we have that $\mathrm{supp}\, f^* = (0,\infty)$. Then we have $\mathrm{supp}\,v = \mathrm{supp}\,f$ and again  $v^*=w$. In fact the conclusions of  (a) remain valid when defining  $h\circ\tau(t)=0$ for any $t\not\in \mathrm{supp}\,f = \mathrm{supp}\,v$.
Thus the conclusion $(f^*/w)^{*,w}=(f/v)^{*,v}$ remains valid provided we define $(f/v)(t)=0$ for $t\not\in \mathrm{supp}\,f$.

Finally if $I= (0,\infty)$ and $\lim_{t\to \infty} f^*(t) >0$, then by Lemma \ref{lem:Ryff-3} for every $\eps>0$ there exists   a measure preserving onto transformation $\tau: I\to I$ such that $|f|\le (1+\eps)\, f^*\circ \tau$.
Defining the weight $v=w\circ \tau$, we have $v>0$ on $I$. By  (a), $v^*=w$ and 
\[
\bigg(\frac{f}{v}\bigg)^{*,v} \le \bigg(\frac{(1+\eps)f^*\circ \tau}{v}\bigg)^{*,v}=(1+\eps)\bigg(\frac{f^*\circ \tau}{w\circ \tau}\bigg)^{*,v}= (1+\eps)\bigg(\frac{f^*}{w}\bigg)^{*,w}.
\]
\end{proof}

Given  $0\le v\in L^0(I)$, let  us introduce some notation.  Set
  \[
  V(t)=\int_0^t v\, dm \ \ \   \hbox{ and assume} \ \ \  v^* = w, \ \ \ \  V(t)<\infty, \ \ \ t\in I. 
  \]
Then $V$  is an increasing, not necessarily strictly increasing,   and  continuous  function  from  $I$ onto  $J=(0, b)$ since $V(a) = \int_0^a v^*\, dm = \int_0^a w \, dm = W(a) = b$.  For  $t\in J $, the set  $V^{-1}\{t\}$ is a closed subinterval of $I$. Let
\[
N_v = \{t\in J: m(V^{-1}\{t\}) > 0 \}.
\]
 Clearly the set $N_v$ is finite or countable. If $t\in N_v$ then $v$ vanishes a.e. on $V^{-1}\{t\}$. If $t\not\in N_v$ then $V^{-1}(t)$ is defined unambiguously as the unique element in $V^{-1}\{t\}$.

For $f\in L^0$ with $\mathrm{supp}\, f\subset\mathrm{supp}\, v$ define $f\circ V^{-1}$ by
\begin{equation}\label{eq:366}
f\circ V^{-1}(t)=
\begin{cases}
0& \hbox{if } t\in N_v,\\
f(V^{-1}(t))& \hbox{if } t\in J \setminus N_v.  
\end{cases}
\end{equation}
With the convention (\ref{eq:366}) above 
the submajorization result of Theorem \ref{H-L} as well as Proposition \ref{lem:equivalence} may be restated in a more transparent way when the weight $w$ is such that $W(t) < \infty$ for all $t\in I$.
 
\begin{corollary}\label{approx}
If  $W < \infty$ on $I$, then for any $v\in L_+^0(I)$ with $v^*=w$,  and every $f\in L_1+L_\infty(I)$ with $\mathrm{supp}\,f\subset \mathrm{supp}\,v$ we have
\[
\frac{f^*}w\circ W^{-1}\prec \frac fv \circ V^{-1}.
\]
Moreover if $I=(0,a)$ with $a<\infty$ or if $I=(0,\infty)$ and $\lim_{t\to\infty}f^*(t) = 0$, then there exists $v\in L_+^0$ with $\mathrm{supp}\,f\subset \mathrm{supp}\,v$ such that $v^*=w$ and
\[
\left(\frac{f^*}w\circ W^{-1}\right)^*= \left(\frac fv \circ V^{-1}\right)^*.
\]
If $I=(0,\infty)$ and $\lim\limits_{t\to\infty}f^*(t)>0$ then for every $\eps>0$ there exists $v>0$ on $I$ such that $v^*=w$ and 
\[ 
\left( \frac fv \circ V^{-1}\right)^* \le (1+\eps)\,\left(\frac{f^*}w\circ W^{-1}\right)^*.
 \]
\end{corollary}

\begin{proof}
  Let $N_v  = \{t_n\}$ be an enumeration of $N_v$ and set $A = \bigcup_{n} V^{-1}\{t_n\}$. Then $A\subset I$ and $\nu(A) = \int_A v\,dm  = 0$. 
If  $t\notin A$ then  $(f\circ V^{-1})\circ V(t)=f(t)$, and so  $(f\circ V^{-1})\circ V = f$ $\nu$-a.e. on $I$.
Moreover for any  $h\in L_+^0$ and $t\ge 0$ by the change of variable formula it holds
\begin{align}\label{equi-meas-compos-V}
m\{h>t\} = \int_I \chi_{(t,\infty)} \circ h \,dm=\int _I \chi_{(t,\infty)}\circ h\circ V d\nu= \nu\{h\circ V >t\}.
\end{align}
It follows that $h$ for $m$ and $h\circ V$ for $\nu$ are equimeasurable. In particular 
\[
m\{|f|\circ V^{-1} > t\} = \nu \{ (|f|\circ V^{-1}) \circ V > t\} = \nu\{|f| > t\},
\]
 and so 
$f\circ V^{-1}$ for $m$ and $f$ for $\nu$ are equimeasurable. Hence $\frac{f}{v}\circ V^{-1}$ for $m$  and $\frac{f}{v}$ for $\nu$ are equimeasurable, and so  
\[
\left(\frac{f}{v} \circ V^{-1}\right)^* = \left(\frac{f}{v}\right)^{*,v}.
\] 
By a similar argument  $\frac{f^*}{w}\circ W^{-1}$ for $m$ and $\frac{f^*}{w}$ for  $\omega$ are equimeasurable as well, and  hence 
\[
\left(\frac{f^*}{w} \circ W^{-1}\right)^* = \left(\frac{f^*}{w}\right)^{*,w}.
\]
Now the conclusion follows directly from  Theorem \ref{H-L} and Proposition \ref{lem:equivalence}.
\end{proof}

\begin{remark}\label{rem:isom-E-E_v}
{\it Let $0\le v\in L^0(I)$ with $V(t)<\infty$ for all $t\in I$, $\nu$ be the measure $v\,dm$ and $J_v=(0,\nu(I))$. Let $E$ be a symmetric space on $J_v$.  Then for every  $h\in E$,  $h\circ V\in E_v$ and   the map $T: h\mapsto h\circ V$ is a  surjective order isometry from $E$ onto $E_v$. }
\end{remark}
\begin{proof}
Indeed by (\ref{equi-meas-compos-V}), $h$ for $m$ and $h\circ V$ for $\nu$ are equimeasurable, thus $T$ embeds isometrically $E$~ into $E_v$. Moreover for every $f\in E_v$ we have $f= T(f\circ V^{-1}) $, thus $T$ is surjective. Here $f\circ V^{-1}$ is defined as (\ref{eq:366}) where $J$ is replaced by $J_v$.
\end{proof}

\section{Spaces $M_{E,w}$}\label{sec:classes-M}

In this section we define a class $M_{E,w}$ of functions contained in $L^0 = L^0(I)$ which will be used later for investigating the K\"othe dual of the Lorentz space $\Lambda_{E,w}$. 

\subsection{Definition and properties}\label{subsec:def-class-M}
Let the class $M_{E,w}$ and the {\it gauge} on $M_{E,w}$ be defined by
\[
M_{E,w} = \bigg\{f\in L^0: \frac{f^*}{w} \in E_w\bigg\} \ \ \text{and}\ \ \ 
\|f\|_{M_{E,w}}=\bigg\| \frac {f^*}w\bigg\|_{E_w} = \bigg\| \bigg(\frac {f^*}{w}\bigg)^{*,w}\bigg\|_{E}.
\]

Although the class $M_{E,w}$ does not need to be even linear it has several properties analogous to those in symmetric spaces, so  a similar terminology is used here as may be seen below.
\begin{proposition}\label{properties class M}
\begin{itemize}
\item[(i)] The class $M_{E,w}$ is a solid symmetric subset of $L^0$, that is $\|f\|_{M_{E,w}} = \|f^*\|_{M_{E,w}}$ and if $f\in L^0$, $g\in M_{E,w}$ and $|f|\le |g|$ a.e. then $f\in M_{E,w}$ and $\|f\|_{M_{E,w}} \le \|g\|_{M_{E,w}}$. 
\item[(ii)] For all $x\in I$, $\chi_{(0,x)}\in M_{E,w}$. Consequently the support of $M_{E,w}$ is equal to the entire interval $I$. 
\item[(iii)] The fundamental function $\phi_{M_{E,w}}(x)= \|\chi_{(0,x)}\|_{M_{E,w}}$, $x\in I$, verifies
\begin{align*}
\phi_{M_{E,w}}(x)\le
2\phi_E(1\wedge b) \left(x+ \frac{1}{w(x)}\right).
\end{align*}
\item[(iv)] If $W<\infty$ on $I$,
then
\[
f\in M_{E,w}\iff \frac {f^*}w\circ W^{-1}\in E\quad \hbox{ and }\quad \|f\|_{M_{E,w}}=\bigg\| \frac {f^*}w\circ W^{-1}\bigg\|_E.
\]
\item[(v)] If $E$ has the Fatou property then the class $M_{E,w}$  has this  property, that is for every $f\in L^0$, $0\le f_n \in M_{E,w}$ with $f_n\uparrow f$ a.e. and $\sup_{n} \|f_n\|_{M_{E,w}}=K<\infty$ we have $f\in M_{E,w}$ and $\|f\|_{M_{E,w}}=K$.
\end{itemize}
\end{proposition}
\begin{proof} (i) It is clear by symmetry and ideal properties of $E_w$.

(ii)   For every $x\in I$ we have
\[
\int_0^x \frac 1w d\omega= \int_0^x \frac 1{w} wdm=x,
\]
thus the function $h_x=\frac 1w\chi_{(0,x)}\in  L_1(I, \omega)$. On the other  hand $h_x \le 1/w(x)$ a.e. equivalently $\omega$-a.e. on $I$, and so it is bounded $\omega$-a.e. on $I$. Hence   $h_x \in L_\infty(I, \omega)$. Consequently  $h_x \in L_1\cap L_\infty(I, \omega)$. Therefore   $h_x^{*,w}\in L_1\cap L_\infty(J, m)$.  Indeed, it is clear that
\begin{align}\label{eq:11}
\|h_x^{*,w}\|_\infty = 1/w(x).
 \end{align}
We also have that $m\{h_x^{*,w} > t\}= \omega\{h_x > t\}$, $t\ge 0$,  in view of equimeasurability of $h_x^{*,w}$ with respect to $m$ on $J$ and $h_x$ with respect to $\omega$ on $I$. Hence 
 \begin{align}\label{ineq:123}
 \|h_x^{*,w}\|_1& = \int_J h_x^{*,w} \, dm = \int_0^\infty m\{h_x^{*,w} > t\} \, dm(t)\\
 & = \int_0^\infty \omega\{h_x > t\}\, dm(t) = \int_I h_x w\, dm= x. \notag
 \end{align}
It is well known \cite{BS, KPS} that   $L_1\cap L_\infty(J,m)\subset E$, and so $h_x^{*,w}={(\chi_{(0,x)}/w)^{*,w}}\in E$. The latter means that $\chi_{(0,x)}\in M_{E,w}$ for every $x\in I$. Thus the support of the space $M_{E,w}$ is the entire interval $I$.

(iii)  Since $E$ is a symmetric Banach function space  it is well known that $\|f\|_E \le C\|f\|_{L_1\cap L_\infty}$, $f\in E$, where  $C= 2\ph_E(1\wedge b)$ (see  \cite{KPS}, Ch. II, Theorem 4.1 and its proof). From (\ref{eq:11}) and (\ref{ineq:123}), $\|h_x^{*,w}\|_{L_1\cap L_\infty} \le x + 1/w(x)$. Thus
\begin{align*}
\phi_{M_{E,w}}(x)&=\left\|h_x\right\|_{E_w} = \left\|h_x^{*,w}\right\|_E\le
2\ph_E(1\wedge b) \left(x+ \frac{1}{w(x)}\right).
\end{align*}

(iv) This condition follows directly from Proposition \ref{prop:S=W}.

(v) It is immediate by the definition of the space $M_{E,w}$ and the properties of the rearrangements. 
\end{proof}

From Theorem \ref{H-L}, Proposition \ref{lem:equivalence} and Corollary  \ref{approx} we obtain directly the next result.

\begin{proposition}\label{inf-formula}  For any $f\in M_{E,w}$ we have
\begin{align*}
\|f\|_{M_{E,w}}&=\inf\left\{ \left\| \frac fv \right\|_{E_v} : v\ge 0, v^*=w, \mathrm{supp}\,v\supset \mathrm{supp}\, f \right\}
\end{align*}
with the convention that $\|g\|_E=\infty$ for every  $g\notin E$, and $f(t)/v(t) = 0$ whenever $f(t) = 0$.

\smallskip
Moreover if $W<\infty$ on $I$, then for $f\in L^0$ we have that $f\in M_{E,w}$ if and only if $ \frac fv \circ V^{-1}\in E$ for some $v\ge 0$ with $v^*= w$ and $\mathrm{supp}\,v\supset \mathrm{supp}\,f$.
\end{proposition}

\begin{remark} The class $M_{E,w}$ does not need to be either linear or normable. Let  $E$ be an Orlicz space $L_\varphi$, then the class $M_{E,w}$ is the class $M_{\varphi,w}$ considered in \cite{KR}. 
In view of \cite[Proposition 3.4]{KR} the class $M_{\varphi,w}$ may not be linear, while by \cite[Proposition 4.14 and Example 4.15]{KR} it may be linear but not normable.
 
\end{remark}

\subsection{Normability}

Before we prove the main result on normability of the class $M_{E,w}$ we need the following lemma.

\begin{lemma}\label{compar-norms}
Let $w_1$, $w_2$ be two decreasing positive weights on $I$ such that for some constant $C\ge 1$ it holds that  $w_1\le Cw_2$ a.e.. Then  for every function $f\in L^0$ we have 
\[
\left(\frac {f}{w_2}\right)^{*,w_2} \prec C \left(\frac {f}{w_1}\right)^{*,w_1}.
\]
  Consequently, if $\int_I w_1dm=\int_I w_2 dm=b$  and $E$ is a fully symmetric space on $J=(0,b)$ then $M_{E,w_1}\subset M_{E,w_2}$  with $\|f\|_{M_{E,w_2}}\le  C\|f\|_{M_{E,w_1}}$  for  $f \in M_{E,w_1}$. 
  \end{lemma}

\begin{proof}
 Setting  $\omega_2= w_2\,dm$, by the well known formula (\cite{BS}, Ch. 2, Proposition 3.3, \cite{KPS}, p.64, (2.14)) we get for $x\in I$, 
\begin{align*}
 \int_0^x \bigg(\frac {f}{w_2}\bigg)^{*,w_2} dm = \sup_{\omega_2(A)\le x}\int_A \frac {|f|}{w_2}d\omega_2= \sup_{\omega_2(A)\le x}\int_A |f|\,dm,
 \end{align*}
 and a similar equation holds true for $w_1$.
Clearly  $w_1\le C w_2$ a.e. implies that  $\sup\limits_{\omega_2(A)\le x}\int_A |f|dm \le \sup\limits_{\omega_1(A)\le Cx}\int_A |f|dm $. Thus
\[
 \int_0^x \bigg(\frac {f}{w_2}\bigg)^{*,w_2}\, dm  \le\int_0^{Cx}   \bigg(\frac {f}{w_1}\bigg)^{*,w_1}\, dm.
  \]
But for  $C\ge 1$, $Cx\in(0,a)$ and a non-negative decreasing function $h$  on $(0,a)$ we have
\[
\int_0^{Cx}h\,dm \le \int_0^xh\,dm + h(x) x (C-1) \le \int_0^x h\,dm + (C-1)\int_0^x h\,dm = C \int_0^x h\, dm,
\]
and the conclusion follows.
\end{proof}

\begin{proposition}\label{prop:normability}
Assume that the weight $w$ is regular that is $W(t) \le Ctw(t)$ for some $C\ge 1$ and all $t\in I$.  Then $M_{E,w}$ is a vector space and the formula
\begin{equation}\label{equiv}
\trnorm{f}:= \inf\left\{\sum\limits_{i=1}^n \|f_i\|_{M_{E,w}}: \sum\limits_{i=1}^n |f_i|\ge |f|\right\}  
\end{equation}
defines a lattice norm  $\trnorm{\cdot}$ on $M_{E,w}$ such that
\begin{equation}\label{compar-gauges}
\trnorm{f} \le \|f\|_{M_{E,w}} \le C\trnorm{f}.
 \end{equation}
Consequently the  class $M_{E,w}$ is a normable vector lattice.
\end{proposition}

\begin{proof}
We will prove that for any finite family $f_1,\dots, f_n$ in $M_{E,w}$ we have 
\begin{equation}\label{eq:0}
\left\|\sum\limits_{i=1}^n f_i\right\|_{M_{E,w}}\le C\sum\limits_{i=1}^n \|f_i\|_{M_{E,w}},
\end{equation}
 where $C$ is the constant of regularity of $w$. Then
$\trnorm{\cdot}$ defined by (\ref{equiv})
is a vector lattice norm on $M_{E,w}$  equivalent to the gauge $\|f\|_{M_{E,w}}$. In fact we will verify  (\ref{compar-gauges}).

We claim that
\begin{equation}\label{claim}
\left(\frac 1w\left(\sum_{i=1}^n f_i\right)^*\right)\circ W^{-1} \prec C\sum_{i=1}^n \left(\frac{f_i}{v_i}\circ V_i^{-1}\right)^*
\end{equation}
for every non-negative functions $v_1,\dots, v_n$ with $\mathrm{supp}\, f_i \subset \mathrm{supp}\, v_i$,  $v_i^*= w$, $i=1,\dots, n$, where $V_i^{-1}$ are defined as in the proof of Corollary \ref{approx}, since
 $V_i(t) = \int_0^t v_i\, dm \le \int_0^t v_i^*\, dm = \int_0^t w\, dm = W(t) <\infty$ for all $t\in I$. 
The statement of the claim then implies the following
\[
\left\|\left(\frac 1w\left(\sum_{i=1}^n f_i\right)^*\right)\circ W^{-1}\right\|_E \leq  C\sum_{i=1}^n \left\|\frac{f_i}{v_i}\circ V_i^{-1}\right\|_E.
\]
Taking the infimum of every right term with respect to  $v_i$ with $v_i^*= w$ and $\mathrm{supp}\, f_i \subset \mathrm{supp}\, v_i$ for $i=1,\dots ,n$, we get by Proposition \ref{inf-formula},
\[
\left\|\left(\frac 1w\left(\sum_{i=1}^n f_i\right)^*\right)\circ W^{-1}\right\|_E \leq  C\sum_{i=1}^n \left\|\frac{f^*_i}{w}\circ W^{-1}\right\|_E,
\]
and consequently in view of Proposition \ref{properties class M}{(iv)} we obtain the desired inequality (\ref{eq:0}).

Now in order to finish it is enough to  prove claim (\ref{claim}), which is equivalent to the following inequality
\begin{equation}\label{ineq-weights}
\int_0^x \left(\frac{\left(\sum_{i=1}^n f_i\right)^*}{w}\circ W^{-1}\right)^*dm\le
C\sum_{i=1}^n \int_0^x\left(\frac{|f_i|}{v_i}\circ V_i^{-1}\right)^*dm, \ \ \  x\in J.
\end{equation}
For any measurable $v\ge 0$ with $V(t) = \int_0^t v\,dm <\infty$, $t\in I$, and $f\in L^0$ such that $\mathrm{supp}\,f \subset\mathrm{supp}\, v$, by equimeasurability of $f/v$ for $d\nu=vdm$ and $(f/v) \circ V^{-1}$ for $m$ we have  that 
$(f/v)^{*,v} = ((f/v)\circ V^{-1})^*$. Hence by  (\ref{eq:1})
for any $x\in J$,
 \begin{align*}
 \int_0^x  \left(\frac fv\circ V^{-1}\right)^*dm&= \int_0^x \left(\frac{f}{v}\right)^{*,v} \, dm =\inf_{\lambda>0}\left\{\int_I\left(\frac{|f|}v-\lambda\right)_+\, d\nu +\lambda x\right\}\\
 & = \inf_{\lambda>0}\left\{ \int_I (|f|-\lambda v)_+dm+\lambda x\right\}.
\end{align*}
Thus the righthand side of (\ref{ineq-weights}) has the following form 
 \begin{equation}\label{eq:2}
R(x):= \sum_{i=1}^n \int_0^x\left(\frac{|f_i|}{v_i}\circ V_i^{-1}\right)^*dm= \inf_{\substack{\lambda_i  > 0\\ i=1,\dots, n}} \left\{ \int_I
\sum_{i=1}^n (|f_i|- \lambda_i v_i)_+dm+ \sum_{i=1}^n\lambda_i x\right\}.
\end{equation}
The function 
$s\mapsto s_+$ is subadditive and non-decreasing on $\mathbb R$. Hence a.e. on $I$,
\begin{align*}
\left(\left|\sum_{i=1}^n f_i\right| - \sum_{i=1}^n \lambda_i v_i \right)_+ &\le \left(\sum_{i=1}^n |f_i| - \sum_{i=1}^n \lambda_i v_i\right)_+
\le \sum_{i=1}^n (|f_i|- \lambda_i v_i)_+.
\end{align*}
Thus by (\ref{eq:2}), in view of (\ref{eq:1}) we get for $x\in J$,
\begin{align*}
R(x)\ge &\inf_{\lambda_1,\dots,\lambda_n>0} \left[\int_I\ \bigg(\bigg|\sum_{i=1}^n  f_i\bigg|- \sum\limits_{i=1}^n\lambda_i v_i\bigg)_+\,dm+x\sum_{i=1}^n \lambda_i\right]\\
&= \inf_{\alpha_1,\dots,\alpha_n>0\atop \sum \alpha_i=1}\inf_{\lambda>0}\bigg[ \int_I \bigg(\bigg|\sum_{i=1}^n  f_i\bigg|
  - \lambda\sum\limits_{i=1}^n \alpha_i v_i\bigg)_+\,dm+\lambda x \bigg] \\
 &  = \inf_{\alpha_1,\dots,\alpha_n>0\atop \sum \alpha_i=1; {v = \sum\alpha_iv_i}}\inf_{\lambda>0}\bigg[ \int_I \bigg(\frac{\bigg|\sum_{i=1}^n  f_i\bigg|}{v} - \lambda\bigg)_+ v \, dm + \lambda x \bigg]  \\
 &= \inf_{v\in \mathrm{conv}(v_1,\dots, v_n)}
\int_0^x  \bigg(\frac {\sum_{i=1}^n f_i}{v}\bigg)^{*,v}\,dm\\
&=\inf_{v\in \mathrm{conv}(v_1,\dots, v_n)}\int_0^x\bigg(\frac {\big|\sum_{i=1}^n f_i\big|}v\circ V^{-1}\bigg)^*\, dm.
\end{align*}
If $v\in \mathrm{conv}(v_1,\dots, v_n)$ we have $v=\sum_{i=1}^n \alpha_i v_i$ for some $\alpha_i\ge 0$ with $\sum_{i=1}^n \alpha_i= 1$. Since by $v_i^*=w$ we have $V_i(t)\le W(t)$ for every $0\le t< a$,  with equality  $V_i(a) = \lim_{t\to a^-} V_i(t) = W(a)=\lim_{t\to a^-} W(t)$, we obtain $V(t)=\sum_{i=1}^n \alpha_i V_i(t)\le \sum_{i=1}^n \alpha_i W(t) = W(t)$ for $t\in I$  with $V(a) = W(a)$,
so that the continuous function $V$ maps $I$ onto $J$, and we may define $V^{-1}$ as in the proof of Corollary \ref{approx}.  We  also have $v^*\prec \sum_{i=1}^n \alpha_i v_i^*=w$, hence
\[
tv^*(t)\le\int_0^t v^*\le W(t)\le Ctw(t),\ \ \ \  t\in I,
\]
by regularity of $w$.
But then for every $v\in \mathrm{conv}(v_1,\dots, v_n)$, letting  $V_*(t): =\int_0^t v^*$, we get for $x\in J$,
 \begin{align*}
\int_0^x  \bigg(\frac {\big|\sum_{i=1}^n  f_i\big|}v\circ V^{-1}\bigg)^*\,dm &\ge \int_0^x  \bigg(\frac {\big(\sum_{i=1}^n  f_i\big)^* }{v^*}\circ V_*^{-1}\bigg)^*\,dm\\
&\ge \frac 1C \int_0^x { \bigg(\frac {\big(\sum_{i=1}^n f_i \big)^*}w \circ W^{-1}\bigg)^*}\,dm = :L(x),
\end{align*}
where the first inequality results from Corollary \ref{approx} with $v^*$, $V_*$ playing  the role of $w$, $W$ respectively,   and the second one by   Lemma \ref{compar-norms} applied to the weights $v^*$ and $w$. Thus $CR(x) \succ L(x)$, and this proves the claim and completes the proof.

\end{proof}

\section{K\"othe duality of $M_{E,w}$.}\label{sec:Koethe duality}
The K\"othe dual  of the class $M_{E,w}$ is defined as  for a Banach function space, as the set of elements $f\in L^0=L^0(I)$ such that
\[ 
 \|f\|_{(M_{E,w})'} := \sup\bigg\{\int_I |fg|\,dm: g\in M_{E,w}, \|g\|_{M_{E,w}}\le 1\bigg\}<\infty. 
\]
The set $(M_{E,w})'$ is an  ideal in $L^0$ on which $f\mapsto \|f\|_{(M_{E,w})'}$ defines a vector lattice norm. Equipped with this norm,  the space $(M_{E,w})'$ becomes a symmetric Banach function space, as it may be shown directly; but this will be also a consequence  of the next theorem.
\begin{theorem}\label{dual-M}
If $W<\infty$ on $I$,  then the K\"othe dual  of $M_{E,w}$ equals $\Lambda_{E',w}$ isometrically, that is $\|f\|_{(M_{E,w})'} = \|f\|_{\Lambda_{E',w}}$. 
\end{theorem}
\begin{proof}
The proof will be done in several steps.

\par a) $\Lambda_{E',w}\subset (M_{E,w})'$ and the inclusion is norm-decreasing i.e. $\|f\|_{(M_{E,w})'} \le \|f\|_{\Lambda_{E',w}}$. Indeed if $f\in \Lambda_{E',w}$ and $g\in M_{E,w}$ then in view of  the assumption $W<\infty$ and Proposition \ref{properties class M} (iv) we get
\begin{align}\label{ineq:11}
\int_I |fg| \,dm&\le \int _I f^*g^*\,dm=\int_I f^*\frac{g^*}{w}wdm=\int_J (f^*\circ W^{-1})\bigg(\frac{g^*}w\circ W^{-1}\bigg)\, dm\\
&\le \|f^*\circ W^{-1}\|_{E'} \bigg\|\frac{g^*}w\circ W^{-1}\bigg\|_E= \|f^*\|_{(E')_w} \bigg\|\frac{g^*}w\bigg\|_{E_w}=
\|f\|_{\Lambda_{E',w}}\|g\|_{M_{E,w}},\notag
\end{align}
which shows  that $\|f\|_{(M_{E,w})'} \le \|f\|_{\Lambda_{E',w}}$. 
\par b) Now we will show that for every $f\in\Lambda_{E',w}$ we get the equality of the norms $\|f\|_{(M_{E,w})'} = \|f\|_{\Lambda_{E',w}}$.  Assume first that $0\le f\in \Lambda_{E',w}$ is decreasing, and so $f\circ W^{-1}$ is also decreasing.  
Then for any $\eps> 0$ we can find a decreasing non-negative function $h\in E$ with  $\|h\|_E = 1$  and satisfying
 \[
 \|f\|_{\Lambda_{E',w}}-\eps = \|f\circ W^{-1}\|_{E'} - \eps \le \int_J (f\circ W^{-1})h\,dm= \int_I f\,( h\circ W) w\,dm.
 \]
Setting $g=( h\circ W)\, w$, we have 
\begin{equation}\label{eq:3}
\int_I fg\, dm\ge \|f\|_{\Lambda_{E',w}}-\eps,
\end{equation}
while $ g/w= h\circ W\in E_w$ with $\|g/w\|_{E_w}= \|h\|_E=1$ by Proposition \ref{prop:S=W}.  Now since $g$ is decreasing we have $g\in M_{E,w}$ and $\|g\|_{M_{E,w}}= \|g/w\|_{E_w} = 1$. Then by (\ref{ineq:11}) and (\ref{eq:3}) we get $\|f\|_{(M_{E,w})'} = \|f\|_{\Lambda_{E',w}}$.

Let us reduce now the general case when $f$ is not decreasing to the preceding one. 

First  assume that  $m(\mathrm{supp}\,f) < \infty$. 
 Then by Proposition \ref{prop:ryff}(i) there exists a measure preserving and onto transformation $\tau$ on $I$ such that $|f| = f^*\circ \tau$. Let $g$ be chosen to satisfy (\ref{eq:3}) for $f^*$   in place of $f$. Then
 \begin{equation}\label{ineq:12}
 \int_I |f| (g\circ \tau)\,dm = \int_I(f^*\circ \tau) (g\circ \tau)\,dm = \int_I f^*g\, dm \ge \|f^*\|_{\Lambda_{E',w}}-\eps= \|f\|_{\Lambda_{E',w}}-\eps,
 \end{equation}
and  $\|g\circ\tau\|_{M_{E,w}}= \|g\|_{M_{E,w}} = 1$.

Now let $m(\mathrm{supp}\,f) = \infty$. There exists a sequence of functions $f_n$ with $m(\mathrm{supp} (f_n)) < \infty$ and such that $|f_n|\uparrow |f|$ a.e.. Hence   $f_n^* \uparrow f^*$ a.e., and by the Fatou property of $\Lambda_{E',w}$ (see Proposition \ref{prop:2}) we get $\|f_n\|_{\Lambda_{E',w}} \uparrow \|f\|_{\Lambda_{E',w}}$. 

Now by (\ref{ineq:12}) for each $f_n$ we can find $g_n\ge 0$ with  $\|g_n\|_{M_{E_w}} = 1$ and such that
\[
\int_I |f_n| g_n \, dm \ge \|f_n\|_{\Lambda_{E', w}} - \frac 1n, \ \ \  \ n\in\mathbb{N}.
\]
Then 
\[ \|f\|_{(M_{E,w})'}\ge \limsup_{n\to \infty} \int_I |f_n| g_n \, dm \ge \lim_{n\to \infty}\left (\|f_n\|_{\Lambda_{E', w}} -\frac 1n\right) =  \|f\|_{\Lambda_{E',w}}.
\]
\par  
c) 
By a) and b) we have that $\Lambda_{E',w}\subset (M_{E,w})'$ and this inclusion is isometric, so $\Lambda_{E',w}$ is a closed ideal in $(M_{E,w})'$. This ideal is order dense, since it contains the bounded functions with finite measure supports, and moreover it has  the Fatou property. It follows that $\Lambda_{E',w}$ is  equal to $(M_{E,w})'$. In fact if $0\le f\in (M_{E,w})'$   there exists a sequence $(f_n)\subset  \Lambda_{E',w}$  with $0\le f_n\uparrow f$ a.e.. Moreover  $\|f_n\|_{\Lambda_{E',w}} = \|f_n\|_{(M_{E,w})'}\le \|f\|_{(M_{E,w})'}$. Then by the Fatou property of $\Lambda_{E',w}$, $f\in \Lambda_{E',w}$.
\end{proof}

The next result is a generalization of \cite[Theorem 2(i)]{HKM}.

\begin{corollary}
Let $W<\infty$ on $I$. If $E$ has the Fatou property and $w$ is regular, then  $(\Lambda_{E',w})'=M_{E,w}$ as sets  with the gauge $\|\cdot \|_{M_{E,w}}$ equivalent to the norm $\|\cdot \|_{(\Lambda_{E',w})'}$.
\end{corollary}

\begin{proof}
It is well known that a Banach function lattice $F$ has the Fatou property if and only if $F=F''$ isometrically \cite{LT2, Z}. The gauge $\|\cdot\|_{M_{E,w}}$ is not a norm, but it is equivalent to a lattice norm on $M_{E,w}$ by Proposition \ref{prop:normability}. Moreover  by Proposition \ref{properties class M} the class $(M_{E,w},\|\cdot\|_{M_{E,w}})$ has the Fatou property.  Now analogously as in the proof of Theorem 1, page 470 in \cite{Z}, or page 30 in \cite{LT2} one can show that
\[
(M_{E,w})''=M_{E,w}\ \ \ \text{as sets, \  and} \ \ \ \|\cdot\|_{(M_{E,w})''} \ \ \ \text{is equivalent to} \ \ \ \|\cdot\|_{M_{E,w}}.
\] 
Then by Theorem \ref{dual-M} we get the equality of sets
$ M_{E,w}=(M_{E,w})'' =(\Lambda_{E',w})'$ with equivalence of $\|\cdot\|_{M_{E,w}}$ and  $\|\cdot\|_{(\Lambda_{E',w})'}$.

\end{proof}

\section{Spaces $Q_{E,w}$}\label{sec:Q}
In this chapter we introduce  a new space related to the class
$M_{E,w}$.

\subsection{Definition and properties}

\begin{definition}\label{def:Qspace}
We denote by $Q_{E,w}$ the set of elements of $L^0=L^0(I)$ which are submajorized by  elements of $M_{E,w}$. For $f\in Q_{E,w}$ we set
\[
 \|f\|_{Q_{E,w}}=\inf\{ \|g\|_{M_{E,w}}: f\prec g \}.
\]
\end{definition}
Given a positive and decreasing weight  $w$ on $I$ and assuming that $W<\infty$, recall that the {\em Marcinkiewicz function space} $M_W$ is  defined as 
\[
 M_W = \left\{f\in L^0: \|f\|_{M_W} = \sup_{x\in I} \frac{\int_0^x f^*}{W(x)} < \infty \right\},  
 \]
and  the space $L_1 + M_W$ is the set of all functions $f\in L^0$ such that 
\[
\|f\|_{L_1 + M_W} = \inf\{\|h\|_1 + \|g\|_{M_W}:\ f= h+g, \ h\in L_1, \ g\in M_W\} < \infty.
\] 
 The spaces $(M_W, \|\cdot\|_{M_W})$ and $(L_1 + M_W, \|\cdot\|_{L_1 + M_W})$  are fully symmetric spaces  \cite{BS, KPS}.  

\begin{theorem}\label{class Q}
Let $w$ be a weight function such that $W<\infty$ on $I$.
\begin{itemize}
\item[(i)] The class $Q_{E,w}$ is a solid linear subspace of $L_1+M_W$ such that  
\begin{equation}\label{inclusion-Q}
\|f\|_{L_1 + M_W} \le C \|f\|_{Q_{E,w}}\ \ \ \text{with}\ \ \ C\le {(1\wedge b)/\phi_E(1\wedge b)}.
\end{equation}
\item[(ii)]  The functional $\|\cdot\|_{Q_{E,w}}$ is a norm on $Q_{E,w}$.
\item[(iii)]  $Q_{E,w}$ equipped with the norm $\|\cdot\|_{Q_{E,w}}$, is the smallest  fully symmetric Banach function space  containing the class $M_{E,w}$.
\item[(iv)] We have   $(Q_{E,w})'= \Lambda_{E',w}$ with equality of norms.
\end{itemize}
\end{theorem}
\begin{proof}
(i) If $f\in M_{E,w}$ then ${f^*/ w}\in E_w$. The space $E_w$ is fully symmetric with respect to the measure $d\omega = wdm$ on $I$ by Proposition \ref{prop:S=W}, so  $E_w \hookrightarrow (L_1+L_\infty)(I,\omega)$ with the embedding constant $C\le {1\wedge b\over\phi_E(1\wedge b)}$ by \cite[Ch. II, Theorem 4.1]{KPS} and the fact that $E$ and $E_w$ have the same fundamental function. 
 Since $w$ is positive, the norms in $L_\infty(I,\omega)$ and $L_\infty(I)$ are equal. Thus  for any $\epsilon >0$ there exist $g\in L_1(I,\omega), h\in L_\infty(I,\omega)$ such that
 \[
 {f^*/ w}=g+h \hbox{ and }
\|g\|_{L_1(I,\omega)}+\|h\|_\infty\le C\|{f^*/ w}\|_{E_w} + \epsilon =C\|f\|_{M_{E,w}} + \epsilon.
\]
Then $f^*=gw+ hw$, $\|gw\|_{1}=\|g\|_{L_1(I,\omega)}$ and $\|hw\|_{M_W}\le \|h\|_\infty \|w\|_{M_W}=\|h\|_\infty$. Hence 
\[
\|f^*\|_{L_1+M_W}\le \|gw\|_{1} + \|hw\|_{M_W} \le \|g\|_{L_1(I,\omega)}+\|h\|_\infty\le C\|f\|_{M_{E,w}} + \epsilon,
\]
which gives $\|f\|_{L_1+M_W}\le C\|f\|_{M_{E,w}}$ for any $f\in M_{E,w}$.

Assume now that $f\in Q_{E,w}$ and choose $g\in M_{E,w}$ such that $f\prec g$ and $\|g\|_{M_{E,w}}\le (1+\eps) \|f\|_{Q_{E,w}}$. Since $L_1+M_W$ is fully symmetric and by the previous paragraph $g\in L_1 + M_W$, we have $f\in L_1+M_W$ and
\[
\|f\|_{L_1+M_W}\le \|g\|_{L_1+M_W}\le C \|g\|_{M_{E,w}}\le C(1+\eps)\|f\|_{Q_{E,w}}.
\]
  Letting then $\eps\to 0$, we obtain (\ref{inclusion-Q}). 
It is also clear  that $Q_{E,w}$ is a solid subset  in $L_1+M_W$.

(ii) By  (\ref{inclusion-Q})  we have that $\|\cdot\|_{Q_{E,w}}$ is faithful, that is $\|f\|_{Q_{E,w}}=0$ implies $f=0$ a.e.. Since the homogeneous property of $\|\cdot\|_{Q_{E,w}}$ is clear, we need only to show the triangle inequality. For any $\epsilon > 0$ and
 $f_1,f_2\in Q_{E,w}$,  choose $g_1,g_2\in M_{E,w}$ with
\[
f_i\prec g_i \hbox{ and } \|g_i\|_{M_{E,w}}\le (1+\eps) \|f_i\|_{Q_{E,w}}, \ \ \  i=1,2.
 \]
Then
\[
 (f_1+f_2)\prec f_1^*+f_2^*\prec g_1^*+g_2^*.
 \]
Since $g_i^*/w\in E_w$, $i=1,2$, and $E_w$ is a linear space, we have $(g_1^*+g^*_2)/w \in E_w$ and so $g_1^* + g_2^* \in M_{E,w}$. Thus $f_1+f_2\in Q_{E,w}$. Moreover, since $E_w$ is a normed space we get
\[
\|g_1^*+g_2^*\|_{M_{E,w}} = \|(g_1^*+g_2^*)/w\|_{E_w}\le \|g_1^*/w\|_{E_w}+\|g_2^*/w\|_{E_w} = \|g_1\|_{M_{E,w}} + \|g_2\|_{M_{E,w}}.
 \]
Thus
\[
 \|f_1+f_2\|_{Q_{E,w}} \le \|g_1^* + g_2^*\|_{M_{E,w}} \le\|g_1\|_{M_{E,w}}+\|g_2\|_{M_{E,w}} \le (1+\eps) (\|f_1\|_{Q_{E,w}}+\|f_2\|_{Q_{E,w}}).
  \]
Letting $\eps\to 0$ we obtain that the homogeneous functional $\|\cdot\|_{Q_{E,w}}$ is subadditive, and  thus it is a norm on $Q_{E,w}$.

(iii) By definition of  $\|\cdot\|_{Q_{E,w}}$,  if $f\prec g$, $f\in L^0$ and $g\in Q_{E,w}$ then $f\in Q_{E,w}$ and  $\|f\|_{Q_{E,w}}\le\|g\|_{Q_{E,w}}$. Clearly $\|f^*\|_{Q_{E,w}} = \|f\|_{Q_{E,w}}$. Hence $Q_{E,w}$ is fully symmetric. To prove that $Q_{E,w}$ is complete, by the Riesz criterion it is sufficient to show that if $(f_n)$ is a non-negative sequence in $Q_{E,w}$ with $\sum\limits_{n=1}^\infty\|f_n\|_{Q_{E,w}}<\infty$ then the series $\sum\limits_{n=1}^\infty f_n$ converges in $Q_{E,w}$. In view of completeness of $L_1 + M_W$ and  (\ref{inclusion-Q}), $\sum\limits_{n=1}^\infty f_n$ converges in $L_1+M_W$. 

For every $n$ choose $g_n\in M_{E,w}$  with   $\|g_n\|_{M_{E,w}}\le (1+\eps)
\|f_n\|_{Q_{E,w}}$ and $f_n\prec g_n$. Then $\sum\limits_{n=1}^\infty \|g_n\|_{M_{E,w}}\le (1+\eps) \sum\limits_{n=1}^\infty\|f_n\|_{Q_{E,w}}<\infty$, and since $\|g_n\|_{M_{E,w}}=\|g_n^*/w\|_{E_w}$ it follows that   $\frac 1w\sum\limits_{n=1}^\infty g_n^*$ converges in the Banach function space $E_w$. Therefore 
\[
\left\|\sum\limits_{n=1}^\infty g_n^*\right\|_{M_{E,w}} =\left\|\frac 1w\sum\limits_{n=1}^\infty g_n^* \right\|_{E_w} \le \sum_{n=1}^\infty \left\|\frac{g^*_n}{w}\right\|_{E_w}\le (1+\eps) \sum\limits_{n=1}^\infty\|f_n\|_{Q_{E,w}}.
\]
On the other hand  $\sum\limits_{n=1}^\infty f_n\prec \sum\limits_{n=1}^\infty g_n^*$, thus $\sum\limits_{n=1}^\infty f_n\in Q_{E,w}$ and by the above $\left\|\sum\limits_{n=1}^\infty f_n\right\|_{Q_{E,w}}\le \left\|\sum\limits_{n=1}^\infty g_n^*\right\|_{M_{E,w}}\le (1 + \varepsilon)\sum\limits_{n=1}^\infty\|f_n\|_{Q_{E,w}}$. Letting $\eps\to 0$ we obtain
$\left\|\sum\limits_{n=1}^\infty f_n\right\|_{Q_{E,w}}\le \sum\limits_{n=1}^\infty\|f_n\|_{Q_{E,w}}$.

Similarly for every $m\in \mathbb{N}$ we have $\Big\|\sum\limits_{n=m}^\infty f_n\Big\|_{Q_{E,w}}\le \sum\limits_{n=m}^\infty \|f_n\|_{Q_{E,w}}\to 0$ when $m\to\infty$ and thus $\sum\limits_{n=1}^\infty f_n$ converges in $Q_{E,w}$, which achieves the proof of the completness of $Q_{E,w}$.

Finally if $F$ is a fully symmetric Banach function space containing $M_{E,w}$, it contains also any function that is submajorized by a function of $M_{E,w}$, that is, it contains $Q_{E,w}$, which shows that $Q_{E,w}$ is the smallest  fully symmetric Banach function space  containing the class $M_{E,w}$.

(iv) In view of the assumption $W<\infty$, by Theorem  \ref{dual-M} it is enough to show that the K\"othe dual spaces $(Q_{E,w})'$ and $(M_{E,w})'$ are equal as sets with equal norms. Since $M_{E,w}\subset Q_{E,w}$, and the norm in $Q_{E,w}$ is clearly smaller than the gauge in $M_{E,w}$, the reverse inclusion $(Q_{E,w})'\subset (M_{E,w})'$ holds for their K\"othe duals and $\|h\|_{(M_{E,w})'} \le \|h\|_{(Q_{E,w})'}$.

Conversely if $h\in (M_{E,w})'$, $f\in Q_{E,w}$ and $\eps>0$, let us choose $g\in M_{E,w}$ with $f\prec g$ and
$\|g\|_{M_{E,w}}\le (1+\eps)\|f\|_{Q_{E,w}}$.  Then
\begin{align*}
\int_I |fh|\,dm&\le\int_I f^*h^*\,dm\quad \hbox{ (Hardy-Littlewood inequality \cite[Theorem 2.2]{BS}) }\\
& \le \int_I g^*h^*\,dm\quad  \hbox{ (Hardy's lemma \cite[Proposition 3.6]{BS}, $f^*\prec g^*$, $h^*$ is decreasing) }\\
& \le \|g^*\|_{M_{E,w}} \|h^*\|_{(M_{E,w})'} \le (1+\eps) \|f\|_{Q_{E,w}} \|h\|_{(M_{E,w})'}.
\end{align*}
Letting $\eps\to 0$ we obtain that $h\in (Q_{E,w})'$ with  $\|h\|_{(Q_{E,w})'} \le \|h\|_{(M_{E,w})'}$, and so $\|h\|_{(M_{E,w})'} = \|h\|_{(Q_{E,w})'}$. 
\end{proof}

\subsection{Link with Halperin's level functions}

In this section let $w$ be a positive decreasing  weight function on $I$ such that $W<\infty$ on $I$. 
For $f=f^{\ast }$ locally integrable on $I$, define after Halperin \cite{Ha53}
for $0\leq \alpha<\beta<\infty $, $\alpha,\beta\in I=(0,a)$, $a\le \infty$,
\begin{equation*}
W(\alpha,\beta) =\int_{\alpha}^{\beta}wdm ,\ \ \ F(\alpha,\beta) =\int_{\alpha}^{\beta}f \,dm ,\
\ \ R(\alpha,\beta) =\frac{F(\alpha,\beta) }{W(\alpha,\beta) } ,\ \
\end{equation*}%
and for $\beta =\infty $,
\begin{equation*}
R(\alpha,\beta)=R(\alpha,\infty )=\limsup_{t\rightarrow \infty }R(\alpha,t) .
\end{equation*}%
Then $(\alpha,\beta) \subset I$ is called a \textit{
level interval (resp. degenerate level interval) of $f$ with respect to $w$} if $\beta<\infty $ (resp. $\beta=\infty $%
) and for each $t\in (\alpha,\beta) $,
\begin{equation*}
R(\alpha,t) \leq R(\alpha,\beta) \text{ and }0<R(\alpha,\beta) .
\end{equation*}%
Level intervals can be equivalently assumed to be open, closed or half-closed. If a level interval is not contained in any larger level interval, then it is called \textit{maximal level interval of $f$ with respect to $w$}, or just maximal level interval and in short m.l.i..  In \cite{Ha53}, Halperin proved that maximal level
intervals of $f$ with respect to $w$ are pairwise disjoint and unique and therefore there is at most countable number of maximal level intervals.
\begin{definition}\label{def:level}
\cite{Ha53} Let $f\in L^{0}$ be non-negative, decreasing and locally integrable on $I$.
Then the \textit{level function} $f^{0}$ of $f$ with respect to $w$ is
defined as
\begin{equation*}
f^{0}\left( t\right) =\left\{
\begin{array}{cc}
R(\alpha,\beta)\,w\left( t\right) & \hbox{if $t$ belongs to some maximal level interval $(\alpha,\beta)$},
\\
f\left( t\right) & \text{otherwise}. \hfill
\end{array}%
\right.
\end{equation*}%
For a general $f\in L^0$, $0\le \alpha < \beta < \infty$, $\alpha, \beta \in I$, we define 
\[
f^0=(f^*)^0,\ \ \  F(\alpha, \beta) = \int_\alpha^\beta f^*\,dm, \ \ \ \text{and}\ \ \  F(t) =\int_0^t f^*\, dm, \ \ \ t\in I.
\]
\end{definition}

\begin{fact}[{\bf Properties of level functions}]\label{LF-prop}
Let $f\in L_1 + L_\infty$ and $w$ be a decreasing locally integrable  weight function on $I$.
\begin{itemize}
\item[(i)] \cite[Theorem 3.6]{Ha53} $f^0/w$ is decreasing. Consequently in view of  $w$ being decreasing,  $f^0$ is decreasing as well.
\item[(ii)] \cite[Theorem 3.2]{Ha53} $f\prec f^0$. Moreover if $x$ does not belong to a m.l.i., $\int_0^x f^0\,dm= \int_0^x f^*\,dm$, and so if $I$ is finite, $\int_I f^0\,dm= \int_I f^*\,dm$.
\item[(iii)]  \cite[Theorem 3.7]{Ha53} If $f\prec g$ then $f^0\prec g^0$.
\end{itemize}
\end{fact}

\begin{remark}\label{degenerate-level-function}
(1)  If $I=(0,a)$ with $a<\infty$ then  for every $f\in L_1$, $\|f\|_1 = \|f^0\|_1$ by (ii) in Fact \ref{LF-prop}. Therefore $f^0(t) < \infty$ for $t\in (0,a)$. 

(2)  If $I=(0,\infty)$ there exist functions $f\in L_1+L_\infty$ with a degenerate level function, that is $f^0\equiv \infty$ on $I$. Indeed, consider $f\equiv 1$ on $I$, then  $R(0,t)={t/ W(t)}$ is increasing. Hence $(0,\infty)$ is a m.l.i. of $f$, and if $\lim\limits_{t\to \infty} {t/W(t)}=\infty$ then $R(0,\infty)=\infty$, and so $f^0=R(0,\infty)\cdot w\equiv \infty$. 

 Note that if an interval $(a,\infty)$ with $a>0$ is  a m.l.i. of a function $f$ then $R(a,\infty)<\infty$ since $f^0(a)<\infty$ and $f^0$ is decreasing. Thus the only possible degenerate level function is identically equal to $\infty$ on $I=(0,\infty)$. For $f^0$ to be degenerate it is necessary and sufficient that $\limsup\limits_{t\to\infty}{F(t)/ W(t)} = \infty$.

(3)  When $I=(0,\infty)$ there are two simple cases where $f^0$ is non-degenerate.

\vspace{2mm}

(3a) Let $f\in L_1$.  Then $\lim\limits_{t\to\infty}{ F(t)/ W(t)} =\lim\limits_{t\to\infty} {(\int_0^t f^*) / W(t)} = {\|f\|_1/ W(\infty)} <\infty$.

If $W(\infty) = \infty$ and $(a,\infty)$, $a> 0$, is a m.l.i. of $f$, then $R(a,\infty) = 0$ and so $R(a,t) \le R(a,\infty)= 0$ for all $t>a$. Hence $f^*(t) = 0$ for $t>a$, and so 
$\|f\|_1 = \|f^0\|_1$ by (ii) in Fact \ref{LF-prop}, and consequently $f^0<\infty$ on $(0,\infty)$ and so  $f^0$ is non-degenerate.

If $W(\infty) < \infty$ and  
if $f$ has an infinite m.l.i. say $(a,\infty)$ with $a>0$, then for $t>a$ we have 
$f^0(t) = R(a,\infty) w(t) = \frac{F(a,\infty)}{W(a,\infty)}w(t)  < \infty$.  Clearly $f^0(t) < \infty$ for $t\in (0,a)$, and so $f^0$ is non-degenerate.   Moreover $\|f\|_1 = \|f^0\|_1$.

\vspace{2mm}

(3b) Let  $f\in M_W$. Then by definition we have $f\prec C w$ where $C=\|f\|_{M_W}$. Hence $f^0\prec C w^0$ by  (iii) of Fact \ref{LF-prop}. But $w^0=w$, and so $\int_0^t f^0 dm\le CW(t)$ for $t\in I$. Thus $f^0\in M_W$  with $\|f^0\|_{M_W}\le \|f\|_{M_W}$. Therefore $f^0$ is non-degenerate. In addition by $f\prec f^0$ we have  $\|f\|_{M_W}\le \|f^0\|_{M_W}$, and it follows the equality of norms $\|f\|_{M_W}= \|f^0\|_{M_W}$.
\end{remark}

\medskip
The above two simple cases (3a) and (3b) may be combined as follows.
\begin{lemma}\label{stability}
 $f\in L_1+M_W$ if and only if  $f^0\in L_1 + M_W$, and $\|f\|_{L_1 + M_W} = \|f^0\|_{L_1 + M_W}$.
\end{lemma} 
\begin{proof}
 Assume $\|f\|_{L_1+M_W} <1$. We have $f=g+h$ with some $g\in L_1$, $h\in M_W$ such that $\|g\|_1+\|h\|_{M_W}<1$. Then $f^*\prec g^*+h^*\prec g^*+\|h\|_{M_W}w$. It follows that $f^0\prec (g^*+\|h\|_{M_W}w)^0$. It is easy to see that $g^*+Cw$ and $g^*$ have the same m.l.i. and that $(g^*+Cw)^0=g^0+Cw$, $C = \|h\|_{M_W}$. Then 
 \[
 \|f^0\|_{L_1+M_W}\le \|g^0 + Cw\|_{L_1 + M_W} \le \|g^0\|_1+\|h\|_{M_W}\|w\|_{M_W}=\|g\|_1+\|h\|_{M_W}<1.
 \]
  This shows that $\|f^0\|_{L_1+M_W}\le \|f\|_{L_1+M_W}$ for every $f\in L_1+M_W$. The converse inclusion and inequality follow from $f\prec f^0$.
\end{proof}

\begin{notamark}\label{notamark}
If $g, h\in L^0$ then we  write $g\prec_w h$ if $g^{*,w}\prec h^{*,w}$. Clearly if $h\in E_w$ and $g\prec_w h$ then $g\in E_w$ and $\|g\|_{E_w}\le \|h\|_{E_w}$. 
\end{notamark}
\begin{lemma}\label{M level-closed}
For $f\in M_{L_1+L_\infty,w}$ we have $\ds \frac{f^0}w\prec_w \frac {f^*}w$.
\end{lemma}
\begin{proof}

Note that the hypothesis $f\in M_{L_1+L_\infty,w}$ is the right one for ensuring that $f^*/w$ is locally integrable in measure $\omega$, that is integrable on every set of finite measure $\omega$. It implies also that $f^*\in L_1+M_{L_\infty,w}\subset L_1+M_W$  (see Example \ref{ex:9.2}), thus by Lemma \ref{stability}, the level function  $f^0<\infty$ belongs to $L_1+M_W$ too.

By (\ref{eq:1}) we have to prove that for each $x\in J$,
\begin{align}\label{ineq-toprove}
 \inf_{\lambda>0}  \bigg[\int_I \big(f^0-\lambda w\big)_+ \,dm +\lambda x \bigg]
 &\le \inf_{\lambda>0}  \bigg[\int_I \big(f^*-\lambda w\big)_+ \,dm +\lambda x \bigg].  
 \end{align}
If $(\alpha,\beta)\subset I$ is a non-degenerate m.l.i. of $f^*$  we have for any $\lambda >0$,
\begin{align*}
\int_\alpha^\beta \big(f^*-\lambda w\big)_+ \,dm &\ge \bigg(\int_\alpha^\beta \big(f^*-\lambda w\big)\,dm\bigg)_+ \\
&= (F(\alpha,\beta)-\lambda W(\alpha,\beta)\big)_+ = \big(R(\alpha,\beta)-\lambda)_+W(\alpha,\beta)\\
&=\int_\alpha^\beta \big(R(\alpha,\beta)-\lambda)_+ wdm= \int_\alpha^\beta \big(f^0-\lambda w)_+ \,dm.
\end{align*}
Consider now a degenerate m.l.i. $(\alpha,\infty)$ of $f^*$. Since $R(\alpha, t) \le R(\alpha, \infty)$ for all $t \ge \alpha$ and $R(\alpha, \infty) = \limsup_{t\to\infty} R(\alpha,t)$, there exists a sequence $(t_n)$  such that $t_n\uparrow\infty$ with $R(\alpha,t_n)\uparrow R(\alpha,\infty)$. Then as above we have for each  $n\in\mathbb{N}$,
\[
\int_\alpha^{t_n} \big(f^*-\lambda w\big)_+ \,dm \ge \int_\alpha^{t_n} \big(R(\alpha,t_n)-\lambda)_+ wdm.
\]
Passing to the limit $n\to \infty$ we obtain
\begin{align*}
\int_\alpha^\infty \big(f^*-\lambda w\big)_+ \,dm &\ge \int_\alpha^\infty \big(R(\alpha,\infty)-\lambda)_+ wdm\\
&=  \int_\alpha^\infty \big(f^0-\lambda w)_+ \,dm.
\end{align*}
On the complementary set $C$ of the union of all the m.l.i. we have $f^0=f^*$, and thus
\[ \int_C  \big(f^0-\lambda w\big)_+ \,dm = \int_C \big(f^*-\lambda w\big)_+ \,dm. \]
Adding this equality with all the inequalities we obtained on each m.l.i. we get
\[ \int_I  \big(f^0-\lambda w\big)_+ \,dm \le \int_I \big(f^*-\lambda w\big)_+ \,dm, \]
which implies (\ref{ineq-toprove}).
\end{proof}
If $f\in M_{E,w}$ then $f\in M_{L_1 + L_\infty,w}$ and so by Lemma \ref{M level-closed},  $\frac{f^0}{w} \prec_w \frac{f^*}{w}$, and so by Notation \& Remark \ref{notamark},  $\|f^0\|_{M_{E,w}} = \|f^0/w\|_{E_w} \le \|f^*/w\|_{E_w} = \|f\|_{M_{E,w}}$. Thus we get the next result.
\begin{lemma}\label{immed-cor}
If $f\in M_{E,w}$ then $f^0\in M_{E,w}$ and $\|f^0\|_{M_{E,w}}\le \|f\|_{M_{E,w}}$.
\end{lemma}
Now we prove that submajorization for level functions implies an inequality for their gauges in $M_{E,w}$.
\begin{lemma}\label{M submajo-closed}
For $f,g\in L_1+L_\infty$,   $f^0\prec g^0$ if and only if $\ds \frac {f^0}w\prec_w \frac {g^0}w$. Consequently,  if $f\prec g$ and $g^0\in M_{E,w}$ then  $f^0\in M_{E,w}$, and moreover $\|f^0\|_{M_{E,w}}\le \|g^0\|_{M_{E,w}}$.
\end{lemma}
\begin{proof}
Let $f^0\prec g^0$. By Fact \ref{LF-prop}(i) the functions $f^0/w$, $g^0/w$ are both decreasing. Therefore by Proposition \ref{prop:S=W} (i), for $x\in J$,

\begin{align*}
\int_0^x \kern -2pt\left(\frac{f^0}w\right)^{*,w}\kern -3pt dm &= \int_0^x  \kern -2pt \bigg(\frac{f^0}w\circ W^{-1}\bigg)^* \kern -2pt dm\kern -1pt = \kern -2pt  \int_0^x  \kern -2pt\frac{f^0}w\circ W^{-1}  dm = \int_0^{W^{-1}(x)}  \frac{f^0}w wdm \\ & = \int_0^{W^{-1}(x)}  \kern-4pt f^0\,dm
\le  \int_0^{W^{-1}(x)}  \kern-4pt g^0 \,dm = \int_0^x \kern -2pt\left(\frac{g^0}w\right)^{*,w}\kern -3pt dm,
\end{align*}
and so  $\ds \frac {f^0}w\prec_w \frac {g^0}w$. The proof of the opposite implication is similar. 

Now by Fact \ref{LF-prop}(iii) if $f\prec g$ then  $f^0\prec g^0$,
and by the preceding $\ds \frac {f^0}w\prec_w \frac {g^0}w$, which implies that $\|f^0\|_{M_{E,w}} = \|f^0/w\|_{E_w} \le \|g^0/w\|_{E_w} = \|g^0\|_{M_{E,w}}$. 

\end{proof}

\begin{theorem}\label{th:Q-by-LF}
A function $f\in L_1+M_W$ belongs to $Q_{E,w}$ if and only if its level function $f^0$ relative to $w$ belongs to $M_{E,w}$, and then  $\|f\|_{Q_{E,w}}=\|f^0\|_{M_{E,w}}$.
\end{theorem}
\begin{proof}
If $f\in L_1+M_W$ then $f^*, f^0<\infty$ on $I$ and $f\prec f^0$ by Fact \ref{LF-prop} (ii). Thus if $f^0\in M_{E,w}$, then $f\in Q_{E,w}$, and $\|f\|_{Q_{E,w}}\le \|f^0\|_{M_{E,w}}$. Conversely if $f\in Q_{E,w}$, there is $g\in M_{E,w}$ with $f\prec g$, and for any such $g$  we have by Lemmas  \ref{immed-cor} and \ref{M submajo-closed} that $g^0\in M_{E,w}$ and $f^0\in M_{E,w}$ and moreover
\[\|f^0\|_{M_{E,w}}\le \|g^0\|_{M_{E,w}}\le\|g\|_{M_{E,w}}.
\]
It follows that $\|f^0\|_{M_{E,w}}\le \|f\|_{Q_{E,w}}$.
\end{proof}

\begin{proposition}\label{Fatou for Q}
If $E$ has the Fatou property then so has $Q_{E,w}$, and moreover $(\Lambda_{E',w})'=Q_{E,w}$ with equal norms.
\end{proposition}

\begin{proof}
If $Q_{E,w}$ has the Fatou property, then since by Theorem \ref{class Q} we have $\Lambda_{E',w}=(Q_{E,w})'$ with equal norms, it follows that  $(\Lambda_{E',w})'=(Q_{E,w})^{''}=Q_{E,w}$ with equal norms.

It remains to  prove that $Q_{E,w}$ has the Fatou property when $E$ has the property.
Let $f_n\uparrow f$ a.e. with $f_n\in Q_{E,w}$ and $\sup_n \|f_n\|_{Q_{E,w}}=K<\infty$. Since by Theorem \ref{class Q} (i), $\|f_n\|_{L_1+M_W}\le C \|f_n\|_{Q_{E,w}}$  and $L_1+M_W$ has the Fatou property, we have that $f\in L_1+M_W$, and $f_n^*\uparrow f^*$ a.e..

Letting $g_n=f_n^0$,  by Theorem \ref{th:Q-by-LF} we have $g_n\in M_{E,w}$ with $\|g_n\|_{M_{E,w}}=\|f_n\|_{Q_{E,w}}$. Moreover $g_n$ and $f^0$ are decreasing and $f_n\prec g_n\prec f^0$ by Fact \ref{LF-prop}.  Now by Helly's Selection Theorem \cite[Chapter 8, Section 4]{N} we may find a subsequence $(g_{n_k})$ which converges a.e. to some $g$. By Proposition \ref{properties class M}, $M_{E,w}$ has the Fatou property and so
 $g\in M_{E,w}$ with $\|g\|_{M_{E,w}}\le \liminf \|g_{n_k}\|_{M_{E,w}}= \liminf \|f_{n_k}\|_{Q_{E,w}}\le K$. 
 If we show that $f\prec g$ then $f\in Q_{E,w}$, and 
 \[
 \limsup \|f_{n_k}\|_{Q_{E,w}} \le \|f\|_{Q_{E,w}} \le \|g\|_{M_{E,w}} \le \liminf \|f_{n_k}\|_{Q_{E,w}},
 \]
  which shows that $\|f_n\|_{Q_{E,w}} \uparrow \|f\|_{Q_{E,w}}$, the Fatou property of $Q_{E,w}$.
  
  Assume further 
without loss of generality that $g_n\to g$ a.e..  By the monotone convergence theorem we get for  $t\in I$,
\[ 
\lim_n\int_0^t f^*_n\,dm = \int_0^t f^*\,dm. 
\]
For every $\alpha\in I$ there is  $N$ such that for $n>N$ we have  $g_n(\alpha)\le g(\alpha^-)+ 1 <\infty$, and so for all $t\ge \alpha$,  $n>N$, $g_n(t) \le g(\alpha^-) + 1$.  Now by the Lebesgue Dominated Convergence Theorem, for $t\in I$,
\[ 
\lim_n \int_\alpha^t g_n\,dm = \int_\alpha^t g\,dm.
\]
On the other hand by $g_n \prec f^0$,
\[ 
\int_0^\alpha g_n\,dm\le  \int_0^\alpha f^0\,dm =: \eps(\alpha),
\]
where the function $\eps(\cdot)$ depends only on $f$  and $w$, and it is continuous because $f^0\in L_1+M_W$ by Lemma \ref{stability} and thus is not degenerate. Then since $f_n\prec g_n$, for every $t\in I$,  
\[ 
\int_0^t f^*\,dm=\lim_n\int_0^t f^*_n\,dm\le\limsup_n \int_0^t g_n\,dm\le \eps(\alpha)+ \int_\alpha^t g\,dm \le \eps(\alpha)+ \int_0^t g\,dm.
\]
Since $\eps(\alpha)\to 0$ when $\alpha\to 0^+$, we get for all $t\in I$,
\[
\int_0^t f^*\,dm\le \int_0^t g\,dm,
 \]
and we obtain that $f\prec g$ as desired. 
\end{proof}

\begin{remark}
 A shorter proof of Proposition \ref{Fatou for Q} can be given using the fact that the level functions of an increasing sequence of functions form themselves an increasing sequence, and if $f_n$ converge to $f$ a.e. then $f_n^0$ converge to $f^0$ a.e..  This result was given by G. Sinnamon for his version of level functions, in the special case of a uniformly bounded sequence on a right-finite interval \cite{S94}. It may be transferred to Halperin's level functions using the results of \cite{FLM}, in the corresponding special case of functions in $M_{L^\infty, w}$ on a finite interval while the correct frame for our study is that of functions having a  $W$-concave majorant  \cite{FLM}.  The proof presented above  avoids this problem and moreover it uses only Halperin's reference paper \cite{Ha53} for the sake of bibliographical simplicity. 
\end{remark}

Despite that the space $E_w$ is considered over $I$ with the measure $d\omega = wdm$, the space $(E_w)'$ will  always denote its K\"othe dual computed with respect to the Lebesgue measure $m$ on $I$ as it is done below.

\begin{lemma}\label{lem:associate}
For any $f\in (E_w)'$ we have $\|f\|_{(E_w)'} = \left\|\frac{f}{w}\right\|_{(E')_w}$. Moreover $(E_w)'' = (E'')_w$ with equality of norms. 
\end{lemma}
\begin{proof} In view of Proposition \ref{prop:S=W} we get 
\begin{align*}
\|f\|_{(E_w)'} &= \sup \left\{\int_I |f| g \, dm: \|g\|_{E_w} \le 1\right\}
= \sup \left\{\int_I \frac{|f|}{w} g \, wdm: \|g\|_{E_w} \le 1 \right\}\\
&=\sup \left\{\int_J \left(\frac{|f|}{w}\circ W^{-1}\right) \cdot  (g\circ W^{-1}) \, dm: \|g\circ W^{-1}\|_{E} \le 1\right\}\\
&=\sup \left\{\int_J \left(\frac{|f|}{w}\circ W^{-1}\right) \cdot  h \, dm: \|h\|_{E} \le 1 \right\}\\
&=\left\|\frac{|f|}{w}\circ W^{-1}\right\|_{E'} = \left\|\frac{f}{w}\right\|_{(E')_w}.
\end{align*}
This proves the first part. Using this  once for $E$, then for $E'$ we get
\begin{align*}
\|f\|_{(E_w)''} &=  \sup \left\{\int_I |f| h \, dm: \|h\|_{(E_w)'} \le 1,\right\}
= \sup \left\{\int_I |f| \frac{h}{w} \, wdm: \left\|\frac{h}{w}\right\|_{(E')_w} \le 1 \right\}\\
&=\sup \left\{\int_I (|f|w) g \, dm: \left\|g\right\|_{(E')_w} \le 1 \right\} = \left\| {|f|w\over w}\right\|_{(E'')_w} = \|f\|_{(E'')_w},
\end{align*}
which proves the second part.
\end{proof}

\begin{lemma}\label{lem:bidual}
The equality $(\Lambda_{E,w})'' = \Lambda_{E'',w}$ holds with equal norms.
\end{lemma}

\begin{proof}  

 We use the fact that if $F$ is a Banach function space, then $f\ge 0$ belongs to  $F''$ with $\|f\|_{F''}\le 1$ if and only if  there exists a sequence $0\le f_n\uparrow f$ a.e., with $f_n\in F$, $\|f_n\|_F\le 1$ for all $n\in \mathbb{N}$  \cite[Ch. 15, \S 66, Theorem 1]{Z}. 
 
 Assume first that $f\in (\Lambda_{E,w})''$ with norm $\le 1$, and let $0\le f_n\uparrow f$ a.e. with $\|f_n\|_{\Lambda_{E,w}}\le 1$. Then $f_n^*\uparrow f^*$ a.e., and $f_n^*\in E_w$, $\|f_n^*\|_{E_w}\le 1$. Hence $f^*\in (E_w)''$ with
 $\|f^*\|_{(E_w)''} \le 1$. However by Lemma \ref{lem:associate}, $(E_w)''=(E'')_w$  and so  $f\in \Lambda_{E'',w}$ with $\|f\|_{\Lambda_{E'',w}} \le 1$. 
 
  Conversely, let $f\in \Lambda_{E'',w}$ with $\|f\|_{\Lambda_{E'',w}}\le 1$,  then $f^*\in (E'')_w$ with $\|f^*\|_{(E'')_w}\le 1$. Since $(E'')_w=(E_w)''$, there exists $0\le g_n\uparrow f^*\in E_w$, with $\|g_n\|_{E_w}\le 1$. Then 
  \[
  g_n^{*,w}=(g_n\circ W^{-1})^*\uparrow f^*\circ W^{-1}.
  \]
   Setting $h_n=g_n^{*,w}\circ W$, we have  $h_n$ are non-negative and decreasing on $I$. Clearly  $h_n\uparrow f^*$ and $h_n^{*,w}=g_n^{*,w}$, so $\|h_n\|_{\Lambda_{E,w}}=\|h_n\|_{E_w}= \|h_n^{*,w}\|_E = \|g_n^{*,w}\|_E =\|g_n\|_{E_w}\le 1$. Therefore $f^*\in (\Lambda_{E,w})''$ with $\|f\|_{ (\Lambda_{E,w})''} = \|f^*\|_{ (\Lambda_{E,w})''}\le 1$, which shows the desired equality of spaces and norms.
\end{proof}

The next corollary states an important result on K\"othe duality of generalized Lorentz spaces $\Lambda_{E,w}$. As a corollary we obtain a new description of the K\"othe dual space of the Orlicz-Lorentz space (see section \ref{sec:OL} for details).

\begin{corollary}\label{cor:dual-lambda} Let $w$ be a decreasing positive  weight on $I$ and $W<\infty$.  We have  $(\Lambda_{E,w})'=Q_{E',w}$ with equal norms.
\end{corollary}
\begin{proof}
By general theory of Banach function lattices \cite[Theorem 2, p.457]{Z}, $\Lambda_{E,w}$ and its K\"othe bidual $(\Lambda_{E,w})''$ have the same K\"othe duals. The result follows then by applying  Proposition \ref{Fatou for Q} to $E'$ since   $E'$ has the Fatou property, and then Lemma \ref{lem:bidual}.
\end{proof}

As an immediate corollary of Theorem \ref{th:Q-by-LF} and Corollary \ref{cor:dual-lambda} we obtain a generalization of the H\"older-Halperin inequality \cite[Theorem 4.2]{Ha53}.

\begin{corollary}\label{cor:HH-ineq}
Let $w$ be a decreasing positive weight on $I$ and $W<\infty$. For $f\in L^0$ we have
\[\sup\left\{\int_I |fg|:  g\in \Lambda_{E,w}, \|g\|_{\Lambda_{E,w}}\le 1\right\}=
\begin{cases}
\|f^0\|_{M_{E',w}}& \hbox{ if $f^0\in M_{E',w}$},\\
\infty &\hbox{ otherwise}.
\end{cases}
\]
Consequently $\|f\|_{(\Lambda_{E,w})'} = \|f^0\|_{M_{E',w}} = \|f\|_{Q_{E',w}}$ for every $f\in (\Lambda_{E,w})'$.
\end{corollary}

\begin{proof}
 The left member is finite if and only if $f\in (\Lambda_{E,w})'=Q_{E',w}$. In this case  $\|f\|_{(\Lambda_{E,w})'}=\|f\|_{Q_{E',w}} = \|f^0\|_{M_{E',w}}$. Conversely if the right side is finite then $f^0$ is non-degenerate and belongs to $M_{E',w}$.  Thus $f\prec f^0$ implies that $f\in Q_{E',w}= (\Lambda_{E,w})'$.

\end{proof}

\section{Spaces $P_{E,w}$}\label{sec:P}

We assume in this chapter that $W<\infty$ on $I$. 

\begin{definition}\label{def:Pnorm}
We denote by $P_{E,w}$ the union of the classes  $M_{E,v}$, where $v$ is a positive 
decreasing weight submajorized by $w$ on $I$. The symbol $v\downarrow$ means that $v$ is decreasing. This set is equipped with the gauge
\[ 
\|f\|_{P_{E,w}}=\inf\left\{\|f\|_{M_{E,v}}: v>0, v\downarrow\,, v\prec w\right\}.
 \]
\end{definition}

Our goal is to show that $\|\cdot\|_{P_{E,w}}$ is a symmetric norm, and in fact  $P_{E,w}=Q_{E,w}$ as sets and $\|\cdot\|_{P_{E,w}} = \|\cdot\|_{Q_{E,w}}$.
From the next lemma it follows that the gauge on $P_{E,w}$ is faithful.
\begin{lemma}\label{lem:M-P-M}
We have $M_{E,w}\subset P_{E,w}\subset M_{E,\tilde w}$, where $\tilde w(t):={W(t)\over t}$, $t\in I$, and these inclusions are gauge-decreasing.
\end{lemma}
\begin{proof}
The first inclusion and the corresponding gauge inequality are clear. Conversely for each $v\prec w$ we have $tv(t)\le V(t)\le W(t)$, where $V(t) = \int_0^t v\,dm$, $t\in I$. Hence $v(t)\le \tilde w(t)$, $t\in I$,  and in view of Lemma \ref{compar-norms}, $M_{E,v}\subset M_{E,\tilde w}$, with $\|f\|_{M_{E,\tilde w}}\le \|f\|_{M_{E,v}}$. Taking the infimum with respect to $v\prec w$ we obtain $P_{E,w}\subset M_{E,\tilde w}$ with  $\|f\|_{M_{E,\tilde w}}\le \|f\|_{P_{E,w}}$ for $f\in P_{E,w}$.
\end{proof}

\begin{lemma}\label{lem:monot}
If $v$ is a positive decreasing weight such that $v\prec w$ and $h\in E_w$ is decreasing then $h\in E_v$ and $\|h\|_{E_v}\le \|h\|_{E_w}$.
\end{lemma}
\begin{proof}
 By  Hardy's Lemma \cite[Proposition 3.6, p. 56]{BS} since $(h-\lambda)_+$ is decreasing and $v\prec w$, for every $\lambda>0$ we have
\[
\int_I (h-\lambda)_+ \,vdm\le \int_I (h-\lambda)_+ \,wdm.
\]
 Then in view of identity (\ref{eq:1}) for any $x\in J$, 
\begin{align*}
\int_0^x h^{*,v} \,dm &=\inf_{\lambda>0}\left[\int_I (h-\lambda)_+\,vdm+\lambda x\right]\\
&\le \inf_{\lambda>0}\left[\int_I (h-\lambda)_+\,wdm+\lambda x\right]=\int_0^x h^{*,w} \,dm,
\end{align*}
and so $h^{*,v}\prec h^{*,w}$. Thus since $E$ is fully symmetric and $h^{*,w} \in E$ we have that $h^{*,v}\in E$ and so $h\in E_v$. Moreover $\|h\|_{E_v} = \|h^{*,v}\|_E \le \|h^{*,w}\|_E = \|h\|_{E_w}$. 
\end{proof}
\begin{proposition}\label{prop:equal}
We have $(P_{E,w})'=\Lambda_{E',w}$ with equal norms.
\end{proposition}
\begin{proof}
Since $M_{E,w}\subset P_{E,w}$ with gauge decreasing inclusion, we have $(P_{E,w})'\subset (M_{E,w})'=\Lambda_{E',w}$ by Theorem \ref{dual-M}, and the inclusion is  norm decreasing.

Conversely if $g\in P_{E,w}$ and $\eps>0$, there is $v\prec w$ such that 
\[
g\in M_{E,v} \ \ \ \text{and} \ \ \  \|g\|_{M_{E,v}}\le (1+\eps) \|g\|_{P_{E,w}}.
\]
 Let $f\in \Lambda_{E',w}$.  Then $f^*\in (E')_w$ and by Lemma \ref{lem:monot}, $f^* \in (E')_v$, hence $f\in \Lambda_{E',v}$ with $\|f\|_{\Lambda_{E',v}}= \|f^*\|_{(E')_v} \le \|f^*\|_{(E')_w} = \|f\|_{\Lambda_{E',w}}$.
 Then by Theorem \ref{dual-M},  $fg\in L_1$ with 
 \[ 
 \int_I |fg|\,dm\le \|f\|_{\Lambda_{E',v}}\|g\|_{M_{E,v}}\le (1+\eps)\|f\|_{\Lambda_{E',w}}\|g\|_{P_{E,w}}.
 \]
  Thus $f\in (P_{E,w})'$ with $\|f\|_{(P_{E,w})'}\le (1+\eps)\|f\|_{\Lambda_{E',w}}$. Since $\eps > 0$ is arbitrary we obtain that the inclusion $\Lambda_{E',w}\subset (P_{E,w})'$ is norm decreasing.
\end{proof}

 Consider the inverse level function $w^f$  of $w$ with respect to a non-negative decreasing and locally integrable function $f$, that was introduced in \cite[Remark 4.4]{KLR}.  It is defined as 
 \begin{equation*}
w^f\left( t\right) =\left\{
\begin{array}{cc}
f(t)\over R(\alpha,\beta) & \hbox{if $t$ belongs to some maximal level interval $(\alpha,\beta)$},\\
w\left( t\right) & \text{otherwise}. \hfill
\end{array}%
\right.
\end{equation*}%
Comparing this with  Definition \ref{def:level} of $f^0$ we have that $f^0(t) = R(\alpha,\beta) w(t)$ if $t\in (\alpha,\beta)$, and thus 
\begin{align}\label{eq:inverselevel}
w^f(t)=
\begin{cases}\ds
\frac {f(t)}{f^0(t)}w(t) & \hbox{ if } t\in (\alpha,\beta),\\
w(t) & \hbox{ otherwise. }
\end{cases}
\end{align}
By definition of the level function we can show directly that $f(t) > 0$ for $t\in (\alpha,\beta)$. Hence  $f^0 > 0$ on $(\alpha, \beta)$ and since $w$ is  positive on $I$, so $w^f$ is also positive on $I$. Moreover $w^f$ is decreasing and  $w^f \prec w$ \cite[Remark 4.4]{KLR}. For arbitrary $f\in L_1 + M_W$ we define $w^f = w^{f^*}$. Now we are ready to compare the classes $Q_{E,w}$ with $P_{E,w}$. 

\begin{proposition}\label{prop:incl}
$Q_{E,w}\subset P_{E,w}$ and the inclusion is gauge decreasing.
\end{proposition}

\begin{proof}
By Theorem \ref{th:Q-by-LF} we have $\|f\|_{Q_{E,w}}=\|f^0\|_{M_{E,w}}$. 
Clearly $\ds \frac {f^*}{w^f}=\frac {f^0}w$. By Fact \ref{LF-prop}(i) the latter function  is decreasing. Hence by Lemma \ref{lem:monot} and Theorem \ref{th:Q-by-LF} we get
\[
\|f\|_{M_{E,w^f}}= \left\|\frac {f^*}{w^f}\right\|_{E_{w^f}}\le\left\|\frac {f^*}{w^f}\right\|_{E_w}= \left\|\frac {f^0}{w}\right\|_{E_w}=\|f^0\|_{M_{E,w}}=\|f\|_{Q_{E,w}},
\]
and a fortiori $\|f\|_{P_{E,w}}\le \|f\|_{Q_{E,w}}$.
\end{proof}

\begin{remark}
By Lemma \ref{lem:M-P-M} and Proposition \ref{prop:incl} we have $M_{E,w}\subset Q_{E,w}\subset P_{E,w}\subset M_{E,\tilde w}$, with gauge-decreasing inclusions. In particular if $w$ is regular the four classes coincide as sets, and the gauges are equivalent, and we recover that in this case the class $M_{E,w}$ is normable (Proposition \ref{prop:normability}).
\end{remark}
\begin{corollary}\label{cor:77}
If $E$ has the Fatou property then $P_{E,w}=Q_{E,w}$ isometrically, that is $\|f\|_{P_{E,w}}= \|f\|_{Q_{E,w}}$ for every $f\in P_{E,w}$.  Consequently the class $P_{E,w}$ is a fully symmetric Banach function space having all properties discussed in Section \ref{sec:Q}.
\end{corollary}
\begin{proof}
By the Fatou property $E''=E$, and Propositions \ref{prop:equal}, \ref{prop:incl} and Theorem \ref{cor:dual-lambda} we have $Q_{E,w}\subset P_{E,w}\subset (P_{E,w})'' = (\Lambda_{E',w})'=Q_{E'',w}= Q_{E,w}$, and these inclusions are gauge decreasing. Hence $P_{E,w} = Q_{E,w}$ with equality of norms.
\end{proof}
Since $E'$  has the Fatou property we have $P_{E',w} = Q_{E',w}$ by Corollary \ref{cor:77}, and  $Q_{E',w} = (\Lambda_{E,w})'$ by Corollary \ref{cor:dual-lambda}, thus we get the following result which generalizes  \cite[Theorem 2.2]{KLR}, \cite[Corollary 4.12]{KR} from Orlicz-Lorentz to abstract Lorentz  spaces:

\begin{corollary}\label{dual-Pspace}
For any fully symmetric Banach function space $E$, we have  $(\Lambda_{E,w})'=P_{E',w}$ isometrically.
\end{corollary}

Now we investigate the order continuity of spaces $M_{E,w}$ and $P_{E,w}$. 
\begin{proposition}\label{prop:order-continuity}
 If $E$ is an order continuous symmetric space then $M_{E,w}$ and $P_{E,w}$ are order continuous. 
 \end{proposition}

\begin{proof} By Proposition \ref{inf-formula} and the definition of $P_{E,w}$, for each $f\in M_{E,w}$, resp.  $f\in P_{E,w}$, we have
\begin{equation}\label{eq:79}
\|f\|_{M_{E,w}} = \inf\{\|f/v\|_{E_v}: v\in \mathcal V_M\}, \hbox{ resp. } \|f\|_{P_{E,w}} = \inf\{\|f/v\|_{E_v}: v\in \mathcal V_P\},
\end{equation}
where 
\[\mathcal V_M=\{v\in L^0_+, v^*= w, \mathrm{supp}\,v\supset \mathrm{supp}\,f\},\]
\[
\mathcal V_P=\{v\in L^0_+, 0<v^*\prec w, \mathrm{supp}\,v\supset \mathrm{supp}\,f\}.
\]

Since for each $v\in \mathcal V_M$, $v^* = w$, we have  $V(a) = \int_I v\,dm = \int_I w\, dm = b$. Thus $J_v = J$, and 
by Remark \ref{rem:isom-E-E_v} each $E_v$, $v\in \mathcal V_M$, is order isometric with $E$,  so by order continuity of $E$,  $E_v$ is also order continuous. 

 As for $E_v$, $v\in  \mathcal V_P$, we note that  $V(a)\le W(a)=b$, so that  $J_v = (0,V(a))\subset J=(0,b)$. By Remark \ref{rem:isom-E-E_v},  the space $E_v$ is order isometric to the space $\chi_{J_v}E$, a band in $E$, and thus it is  also order continuous.

If $(f_n)$ is a non-negative decreasing sequence in $M_{E,w}$,  respectively $P_{E,w}$, with $f_n\downarrow 0$ a.e., choose $v$  in $\mathcal V_M$, respectively  $\mathcal V_P$, such that  $f_1/v\in E_v$. Then $f_n/v\downarrow 0$ a.e. and thus $\|f_n\|_{E_v} \downarrow 0$. Since $\|f_n\|_{M_{E,w}}\le \|f_n/v\|_{E_v}$, respectively $\|f_n\|_{P_{E,w}}\le \|f_n/v\|_{E_v}$,  we get $f_n\to 0$ in $M_{E,w}$, respectively $P_{E,w}$.
\end{proof}

Recall that the norm of $E$ is $p$-concave for some $1\le p< \infty$, if for some $C>0$, for every $f_i\in E$, $i=1,\dots,n$, $n\in \mathbb{N}$, it holds
 \[
 \left\|\left(\sum_{i=1}^n |f_i|^p\right)^{1/p}\right\|_E 
 \ge C \left(\sum_{i=1}^n\|f_i\|_E^p\right)^{1/p}. 
 \] 
The largest such constant $C$  is called concavity constant of $E$.

Applying the approach as in the proof   of Proposition \ref{prop:order-continuity}, we can show the following statement  about the $p$-concavity of $M_{E,w}$ or $P_{E,w}$, which generalizes \cite[Corollary 3.6]{KR}.

\begin{proposition}
If $E$ is $p$-concave, $1\le p < \infty$, then 
 so are the gauge of $M_{E,w}$ and the norm of $P_{E,w}$, with $p$-concavity constants not exceeding that of $E$. 
\end{proposition}
\begin{proof}

Let $f_i\in M_{E,w}$, $i=1,\dots,n$, and $\epsilon>0$. Then by (\ref{eq:79}) there exists $v\in \mathcal V_M$ such that
\[
L:= \left\|\left(\sum_{i=1}^n |f_i|^p\right)^{1/p}\right\|_{M_{E,w}} + \epsilon \ge  \left\|\left(\frac{1}{v}\sum_{i=1}^n |f_i|^p\right)^{1/p}\right\|_{E_v}.
\]
Since $E_v$ is order isometric to $E$, so the norm $\|\cdot\|_{E_v}$ is also $p$-concave with the same constant, and thus
\begin{align*}
L  &\ge  C\left(\sum_{i=1}^n \left\|\frac{f_i}{v}\right\|^p_{E_v}\right)^{1/p}
\ge C\left(\sum_{i=1}^n \left\|f_i\right\|^p_{M_{E,w}}\right)^{1/p}.
 \end{align*}  
 
 The part on the space $P_{E,w}$ we do analogously applying that $E_v$ is order isometric to the space $\chi_{J_v}E$.

\end{proof}

Recall the definition of the Banach envelope of a quasi-normed linear space $(X, \|\cdot\|_X)$ \cite[pp. 27-28]{KPR}.  Denote by $(X^*, \|\cdot\|_{X^*})$ the dual space to $X$, that is the  space of linear functionals which are bounded with respect to the quasinorm $\|\cdot\|_X$. It is a Banach space equipped with the usual norm $\|\cdot\|_{X^*}$. Let us define a functional on $X$ by
\[ \|x\|_{\widehat X}=\sup\{ |f(x)|: f\in X^*,\hbox{ and } \|f\|_{X^*}\le 1\}. \]
If $X^*$ separates the points of $X$ then $\|\cdot\|_{\widehat X}$ is a norm on $X$. Then the {\it Banach envelope} $\widehat X$ of $X$ is simply the completion of the normed linear space $(X, \|\cdot\|_{\widehat X})$. 
One can show that the Banach envelop of $X$ is the smallest Banach space $(\widehat X, \|\cdot\|_{\widehat X})$ such that 
$\|x\|_{\widehat X} \le \|x\|_X$ for $x\in X$ and $(\widehat X)^* = X^*$.

The following result is a generalization of \cite[Corollary 4.13]{KR}. We refer to section \ref{sec:OL} where the spaces $M_{E,w}$ and $P_{E,w}$ are interpreted in the case of $E$ being an Orlicz space $L_\varphi$. 

\begin{corollary}\label{cor:envelope}
Let $E$ be a fully symmetric order continuous Banach function space with Fatou property. If moreover  $M_{E,w}$ is a linear space and its gauge is a quasinorm, and if $E$ is order continuous, then $P_{E,w}$ is the Banach envelope of $M_{E,w}$. Consequently if $w$ is regular and $E$ is order continuous then $P_{E,w} = M_{E,w}$ with equivalent norm and gauge.
\end{corollary}

\begin{proof}
Clearly $M_{E,w} \subset P_{E,w}$. Since by Proposition \ref{prop:order-continuity}, $M_{E,w}$ and $P_{E,w}$ are both order continuous their topological dual spaces coincide isometrically with their K\"othe duals. By Theorem \ref{dual-M} and Proposition \ref{prop:equal} we have that $(M_{E,w})' = \Lambda_{E',w} = (P_{E,w})'$.  Since  moreover $P_{E,w}$ is a Banach space by Corollary \ref{cor:77}, it must be the Banach envelope of $M_{E,w}$. The second part results from Proposition \ref{prop:normability}.
\end{proof}

\section{Applications to modular and Orlicz-Lorentz spaces}\label{sec:OL}

Here we apply the results obtained in the previous sections to 
 Orlicz spaces $E=L_\ph$. A special feature of these spaces, as well as of Orlicz-Lorentz spaces, is that their Banach space structure is induced by a modular space structure. 
In the present section we introduce modular structures on the spaces $P_{L_\ph,w}$ and $Q_{L_\ph,w}$ by defining two convex modulars P, Q, which have the same domain $\mathcal M_{L_\ph,w}: P_{L_\ph,w}=Q_{L_\ph,w}$. These modulars have the same Luxemburg, resp. Orlicz norms, which are also the norms on $P_{L_\ph,w}$ and $Q_{L_\ph,w}$  when $L_\ph$ is equipped with its Luxemburg, resp. Orlicz norms. The modular P has been already defined in \cite{KR, KLR}. This allows to compare the present work for $L_\ph$ spaces with the results in those papers. The introduction of the modular Q  seems however to be new.

\subsection{Modular spaces}We start with an introduction to modular spaces \cite{M, Nak}. 

\begin{definition}\label{def:modular}
Let $X$ be a real  vector space. For an extended real valued functional    $\rho: X \to [0,\infty]$ consider the following conditions. 

\begin{itemize}
\item[(i)] $\rho(0)=0$ and  $\rho(-x)=\rho (x)$ for every $x\in X$.  
\item[(ii)] If $x\in X$ and $\rho(tx)=0$ for every $t\ge 0$ then $x=0$.
\item[(iii)] $\rho$ is convex.
 \item[(iii')] For every $x\in X$, the extended real valued function $t\to \rho(tx)$ is convex.
\end{itemize}
 If $\rho$ satisfies conditions (i), (iii) then $\rho$ is called a {\it pseudo-modular}, and a {\it modular} if it satisfies also (ii). If $\rho$ fulfills $\rm{(i), (ii), (iii')}$ then $\rho$ will be called a {\it convex along rays-modular} (in short, CAR-modular).  There is also a notion of CAR-pseudo-modular for which (ii) has not to be satisfied. In all preceding cases, the {\it modular domain} $X_\rho$ consists of all $x\in X$ such that  $\rho(tx)<\infty$  for some $t>0$.

Note that in Musielak's  classical terminology  \cite{M}, our `modular' functionals would be called  `convex semi-modular'.
\end{definition}

It is easy to check that for $\rho$ a (pseudo-) modular, $X_\rho$ is a vector space, and
for $\rho$ a CAR-modular it may be only shown to be a symmetric cone.

If $\rho$ is a modular  (resp. a pseudo-modular) then two norms  (resp. semi-norms) on $X_\rho$  are classically associated with $\rho$, which are defined as follows.
\begin{itemize}
 \item[--] the {\it Luxemburg}  (or second Nakano  \cite{Nak}) norm is the Minkowski functional of the convex set
 $ U=\{x\in E: \rho(x)\le 1\}$, thus 
 \begin{equation}\label{eq:8.1}
  \|x\|_\rho=\inf\{\lambda>0: \rho(x/\lambda)\le 1\},
  \end{equation}
 \item[--] the {\it  Orlicz} (or first Nakano  \cite{Nak}) norm is given by Amemiya's formula \cite{M}
\begin{equation}\label{eq:8.2}
  \|x\|_\rho^0=\inf_{\lambda>0} {1+\rho(\lambda x)\over \lambda} =\inf_{t>0} \left(t+t\rho\left(\frac xt\right)\right).
\end{equation}
 \end{itemize}
 There is another expression of the Luxemburg norm, similar to Amemiya's formula. In fact we have 
 \begin{equation}\label{eq:8.3} 
 \|x\|_\rho=\inf_{\lambda>0} {1\vee\rho(\lambda x)\over \lambda} =\inf_{t>0} \left(t\vee t\rho\left(\frac xt\right)\right).
 \end{equation}
Indeed, 
 \[ 
 \inf_{t>0} (t\vee t\rho(t^{-1}x))\ge \inf_{t\ge \|x\|_\rho} t \wedge \inf_{t<\|x\|_\rho} t\rho(t^{-1}x)=\|x\|_\rho\wedge \lim_{t\uparrow\|x\|_\rho}  t\rho(t^{-1}x)=\|x\|_\rho,
 \]
since   by convexity of $\rho$, the map $t\mapsto t\rho(t^{-1} x)$ is decreasing on $(0,\infty)$. On the other hand   
\[
\inf_{t>0} (t\vee t\rho(t^{-1}x))\le \inf_{t > \|x\|_\rho} (t\vee t\rho(t^{-1}x))= \|x\|_\rho
\]
since $\rho(t^{-1}x) \le 1$ for $t> \|x\|_\rho$.

It is clear that a pseudo-modular is a modular if and only if the associated Luxemburg or Orlicz semi-norms are norms.
 
 If we replace  the modular $\rho$ by a CAR-modular then all formulas (\ref{eq:8.1}), (\ref{eq:8.2}) and (\ref{eq:8.3}) remain valid although the functionals $\|\cdot\|_\rho$ and $\|\cdot\|_\rho^0$ are not norms on $X_\rho$ since the triangle inequality may be not satisfied. They are however {\it gauges} that is positively homogeneous functionals.
 
 By (\ref{eq:8.2}) and (\ref{eq:8.3}), the equivalence of $\|\cdot\|_\rho$ and $\|\cdot\|_\rho^0$ is immediate.
 
  \begin{lemma}\label{lem:infmod}
Let $X$ be a vector space and $\rho: X \to [0,\infty]$ be a  (pseudo-, CAR-) modular on $X$. Let $\rho_v: X \to [0,\infty]$, $v\in \mathcal{V}$, be a family of CAR-modulars on $X$.
 If
\[
\rho(x) = \inf_{v\in \mathcal{V}}\rho_v(x),
\]
then   the modular domain of $\rho$ is 
\[ X_\rho=\bigcup_{v\in\mathcal{V}} X_{\rho_v}
\]
and its associated norms are
\[
\|x\|_\rho = \inf\{\|x\|_{\rho_v}: \ v\in \mathcal{V}\} \ \ \text{and} \ \  \|x\|^0_\rho = \inf\{\|x\|^0_{\rho_v}: \ v\in \mathcal{V}\}.
\]
\end{lemma}
\begin{proof}  For $x\in X$ we have 
\[\|x\|_\rho = \inf_{t>0}(t \vee t\rho(t^{-1})) = \inf_{t>0} (t \vee t \inf_{v\in \mathcal{V} }\rho_v(t^{-1}x)) = \inf_{v\in\mathcal{V}} \inf_{t>0} t( 1 \vee \rho_v(t^{-1}x)) = \inf_{v\in\mathcal{V}} \|x\|_{\rho_v}. \]
The formula for Amemiya norm follows analogously.
\end{proof}

\begin{lemma} \label{lem:convexifmod} Let $X\subset L^0(\Omega)$ be a vector space which is closed under rearrangements, i.e. $f^*\in X$ whenever $f\in X$. Assume $\rho: X \to [0,\infty]$  satisfies conditions $\rm{(i), (ii)}$ of Definition \ref{def:modular}, $\rho$ is convex on the cone of decreasing non-negative functions in $X$,  $\rho$ is symmetric that is  $\rho(f^*) = \rho(f)$, and $\rho$ is monotone that is $\rho(f) \le \rho(g)$ if $|f|\le |g|$, $f,g\in X$. Then for $f\in X$,
\[
\bar\rho(f) = \inf\{\rho(g^*): f\prec g, g\in X\}
\]
is a symmetric   pseudo-modular on $X$, monotone with respect to the relation $\prec$, with    associated  Luxemburg and Amemiya  semi-norms given respectively by
\[
\|f\|_{\bar\rho} = \inf\{\|g\|_\rho: f\prec g, g\in X\}\ \ \ \text{and}\ \ \  \|f\|^0_{\bar\rho} = \inf\{\|g\|^0_\rho: f\prec g, g\in X\}.
\]
\end{lemma}
\begin{proof}
It is clear that the   functional $\bar\rho$  satisfies (i) of Definition  \ref{def:modular}, and  is  symmetric and monotone with respect to $\prec$.  Now let  $f_1,f_2\in X_{\bar\rho}$ with $\bar\rho(f_i) < \infty$, $i=1,2$, and $t_1,t_2\ge 0$ with $t_1+t_2=1$. Given $\eps>0$ choose $g_1,g_2\in X$ such that $f_i \prec g_i$ and  $\rho(g_i)\le \bar\rho(f_i)+\eps$, $i=1,2$. Then in view of 
\[t_1f_1+t_2f_2\prec t_1f_1^*+t_2f_2^*\prec t_1g_1^*+t_2g_2^*,\]
we have, by symmetry and convexity of $\rho$ on the cone of decreasing functions  
\begin{align*}
 \bar\rho(t_1f_1+t_2f_2)& 
 \le \rho(t_1g_1^*+t_2g_2^*) 
 \le t_1 \rho(g_1)+t_2 \rho(g_2) 
 \le t_1\bar\rho(f_1)+t_2\bar\rho(f_2)+\eps,
\end{align*}
which shows that $\bar\rho$ is convex.  Since $\rho$ is a CAR-modular, formulas (\ref{eq:8.2}) and (\ref{eq:8.3}) are satisfied. 
Moreover, 
\begin{align*}
\|f\|_{\bar\rho}&=\inf_{t>0} t(1\vee \bar\rho(t^{-1}f))
= \inf_{t>0} t(1\vee \inf_{f\prec g}\rho(t^{-1}g^*))
= \inf_{f\prec g}\inf_{t>0} t(1\vee \rho(t^{-1}g) = \inf_{f\prec g} \|g\|_\rho,
\end{align*}
where $g\in X$. Similarly we get the second formula associated with  Amemiya functional.
\end{proof}

\subsection{Orlicz-Lorentz spaces and their K\"othe duals}\label{sec:OL}

Assume in this section that $W<\infty$ on $I$.
Let $E= L_\varphi$ be an Orlicz space on $J$. As  was mentioned in section \ref{subsec:21}, $L_\varphi$ is a modular space generated by the modular ${I_\varphi(f)}=\int_J\varphi(|f|)\,dm$. Then $E_w = (L_\varphi)_w$ is the set of $f\in L^0(J)$ such that for some $\lambda>0$,  $\int_J\varphi(\lambda |f|) w \, dm < \infty$, so it is a modular space defined by the modular $ \int_J\varphi(|f|)w\,dm$. Hence the generalized Lorentz space $\Lambda_{L_\varphi,w}$ consists of all $f\in L^0(I)$ such that $f^*\in (L_\varphi)_w$, so it is a modular space corresponding to the modular 
\begin{equation}\label{eq:modularOL}
 \Phi(f) :  = \int_I \varphi(f^*)w\,dm.
\end{equation}
This space is usually called an Orlicz-Lorentz space and is denoted by $\Lambda_{\varphi,w}$ \cite{K, KR09, KR}.  Setting now for $f\in L^0=L^0(I)$,
\[
 \mathrm M(f) : =  \int_I \ph\left({f^*\over w}\right)w\,dm,
\]
then the functional M is a CAR-modular on $L^0$. By definition, the space $M_{L_\varphi,w}$ consists of all $f\in L^0$ such that $f^*/w \in (L_\varphi)_w$. It follows that this space is the modular space induced by the CAR-modular $\mathrm{M}$.  Moreover the Luxemburg and Amemiya gauges associated with the modular $M$ on $M_{L_{\ph,w}}$  coincide with those defined  in section \ref{subsec:def-class-M} on $M_{E,w}$ when $E=L_\ph$ is equipped with its Luxemburg and Amemiya norms respectively.

Now we will characterize the spaces $Q_{L_\ph,w}$ and $P_{L_\ph,w}$.

\begin{lemma}\label{lem:P}
Let for $f\in
 L^0$, 
\begin{equation}
\mathrm P(f): =\inf\left\{\mathrm M_v(f):  v\prec w,\ v>0,\ v\downarrow \right\}\ \ \ \text{where} \ \ \  \mathrm M_v(f) = \int_I \varphi\left(\frac{f^*}{v}\right) v\, dm. 
\end{equation}\label{eq:P}
Then $\mathrm P$ is a convex modular  with domain $P_{L_\ph,w}$ and the Luxemburg and Orlicz norms associated with this modular coincide with the norms on $P_{L_\ph,w}$ given by Definition \ref{def:Pnorm}, associated with the Luxemburg and Orlicz  norms respectively on $L_\ph$.
\end{lemma}
\begin{proof}  The modular $\mathrm P$ is convex by  \cite[Theorem 4.7]{KR} and its proof. By convexity of $\varphi$ it is clear that the function $t \mapsto \mathrm M_v(tf)$  is convex for every $f\in L^0$. Therefore $ \mathrm M_v$ is a CAR-modular for every $v>0$.   The last part of the lemma   is a consequence of Lemma \ref{lem:infmod} by letting $\rho(f) = \mathrm P(f)$, $\mathcal{V}=\{v\prec w,\ v>0,\ v\downarrow \}$ and $\rho_v(f) = \mathrm M_v(f)$. 
\end{proof}

\begin{lemma}\label{lem:Q}
Let for $f\in L^0$  
\begin{equation}\label{eq:Q}
\mathrm Q(f): =\inf\left\{\mathrm M(g):  f\prec g, \ g\in M_{L_\ph,w}\right\}.
\end{equation}
Then $\mathrm Q$ is a convex modular  with modular domain $Q_{L_\ph,w}$ and the Luxemburg and Orlicz norms associated with this modular coincide with the norms on $Q_{L_\ph,w}$ given by Definition \ref{def:Qspace}, associated with the Luxemburg and Orlicz norms respectively on $L_\ph$.
\end{lemma}

\begin{proof}  Applying Lemma \ref{lem:convexifmod}, with $\rho(f) = \mathrm M(f)$ and $\bar\rho(f) = \mathrm Q(f)$ gives that Q is a symmetric pseudo-modular, and by Lemma \ref{lem:infmod} its Luxemburg and Orlicz semi-norms coincide with the norms on $Q_{L_\ph,w}$ given by Definition \ref{def:Qspace}, when $L_\ph$ is equipped with its Luxemburg and Orlicz norms, respectively. In particular those semi-norms are in fact norms and Q is a modular.
\end{proof}

The next  fact is well known and can be easily deduced from \cite[Theorem 7.4.1]{G}.
We provide in Appendix a completely different and self-contained proof of it for the convenience of the reader. 

\begin{fact}\label{lem:submaj-phi}
Let $\psi:[0,\infty)\to [0,\infty)$ be a convex increasing function. If $f,g\in L_1+L_\infty\subset L^0(\Omega, \mathcal{A}, \mu)$ with $f\prec_\mu g$ then $\psi(f)\prec_\mu \psi(g)$.
\end{fact}

\begin{proposition}\label{prop:M}
The modular $\mathrm Q(f)$ for $f\in Q_{L_\ph,w}$ is expressed in terms of the level function $f^0$ by
\[
\mathrm Q(f)= \mathrm M(f^0)= \int_I \ph\left(\frac{f^0}w\right)w\,dm.
\]
\end{proposition}
\begin{proof}
Let $f\in Q_{L_\ph,w}$, then for each $g\in M_{L_\ph,w}$ such that $f\prec g$ we have $f^0\prec g^0$ by Fact \ref{LF-prop}(iii). Since $(f^0/w) \circ W^{-1}$ is decreasing, it follows that  $(f^0/w) \circ W^{-1}\prec (g^0/w) \circ W^{-1}$, which by Proposition \ref{prop:S=W} (i)  is equivalent to ${f^0}/w\prec_w  {g^0}/w$. We also have  ${g^0}/w\prec_w  {g^*}/w$ by Lemma \ref{M level-closed}, whence by Fact \ref{lem:submaj-phi} above 
\[
\ph( {f^0}/w)\prec_w \ph({g^0}/w)\prec_w \ph({g^*}/w).
\]
 It follows  that $\mathrm M(f^0)\le \mathrm M(g)$, and so $\mathrm M(f^0)\le \mathrm Q(f)$.  Since $f\prec f^0$ by Fact \ref{LF-prop}(ii),  $\mathrm Q(f)\le  \mathrm M(f^0)$, and the proof is finished. 
 \end{proof}

 In view of Corollary \ref{cor:77}, $Q_{L_\varphi,w}=P_{L_\varphi,w}$, with equal norms, and  we will further use the notation (introduced in \cite{KR}  for the domain of the modular $\mathrm P$)
 \[
 \mathcal{M}_{\varphi,w}: =  Q_{L_\varphi,w}=P_{L_\varphi,w}
 \]

According to whether $L_\ph$ is equipped with its Luxemburg or Orlicz norm, the space $\mathcal{M}_{\varphi,w}$ is equipped with two different norms that we denote by $\|\cdot \|_{\mathcal{M}_{\varphi,w}} $, resp. $\|\cdot \|^0_{\mathcal{M}_{\varphi,w}}$. Each of these norms has two different expressions corresponding to the respective definitions of the norms in $Q_{L_\varphi,w}$ and $P_{L_\varphi,w}$. Moreover by Theorem \ref{th:Q-by-LF} the norm of a function in $Q_{L_\ph,w}$ is the gauge of the corresponding level function in $M_{L_\ph,w}$. We have thus:

\begin{theorem}\label{th:OL} Let $\varphi$ be an  Orlicz function and $w$ be a decreasing positive weight function on $I = (0,a)$, $a\le \infty$, such that $W<\infty$ on $I$. Then for $f\in \mathcal{M}_{\varphi,w}$ we have 
\begin{equation}\label{eq:811}
\|f\|_{\mathcal{M}_{\varphi,w}}  = \inf\{\|f\|_{{\mathrm M}_v}: \ v\prec w, v>0, \ v\downarrow\}  = \inf\{\|g\|_{\mathrm M}; \ f\prec g\} = \|f^0\|_{{\rm M}},
\end{equation}
\begin{equation}\label{eq:812}
\|f\|^0_{\mathcal{M}_{\varphi,w}} = \inf\{\|f\|^0_{{\mathrm M}_v}: \ v\prec w, v>0, \ v\downarrow\} = \inf\{\|g\|^0_{\mathrm M}; \ f\prec g\}  = \|f^0\|^0_{{\rm M}},
\end{equation}
where $\|\cdot\|_{{\mathrm M}}$, $\|\cdot\|_{{\mathrm M}_v}$
 are Luxemburg, and $\|\cdot\|^0_M, \|\cdot\|^0_{M_v}$ are Amemiya gauges.
 \end{theorem}
 
 On the other hand,  $\mathcal{M}_{\varphi,w}$  is the modular space induced by both  modular $\rm Q$ and $\rm P$. To each of the modulars Q, P are associated its Luxemburg and Orlicz norms. It appears  that both the Luxemburg norms of $Q,P$ coincide with $\|\cdot \|_{\mathcal{M}_{\varphi,w}} $, and the Orlicz norms with $\|\cdot \|^0_{\mathcal{M}_{\varphi,w}}$.

 Applying the results developed so far we obtain additional insight on these modular structures. 

\begin{theorem}\label{th:OL*} Let $\varphi$ be an  Orlicz function and $w$ be a decreasing positive weight function on $I = (0,a)$, $a\le \infty$, such that $W<\infty$ on $I$. Then 
\begin{equation}\label{eq:level}
\mathrm Q(f) =  \mathrm M(f^0) = \mathrm M_{w^f}(f)\ge \mathrm P(f).
\end{equation}
For $f\in \mathcal{M}_{\varphi,w}$ we have 
\begin{equation}\label{eq:811a}
\|f\|_{\mathcal{M}_{\varphi,w}} =\|f\|_\mathrm P=\|f\|_\mathrm Q,
\end{equation}
\begin{equation}\label{eq:812a}
\|f\|^0_{\mathcal{M}_{\varphi,w}} =\|f\|^0_\mathrm P=\|f\|^0_\mathrm Q
\end{equation}
 If in addition $\varphi$ is a  $N$-function that is $\lim_{s\to 0} \varphi(s)/s = 0$ and $\lim_{s\to \infty} \varphi(s)/s = \infty$, and either $I$ is finite or  $W(\infty)= \int_0^\infty w\, dm=\infty$, then 
\begin{equation}\label{eq:813}
\mathrm P(f)= \mathrm M(f^0)= \mathrm Q(f).
\end{equation}\vskip-1pt
\end{theorem}

 \begin{proof}
 The first part of  (\ref{eq:level}) follows from Proposition \ref{prop:M} and the second one from equality (\ref{eq:inverselevel}).  Since $w^f\prec w$, $\mathrm M_{w^f}(f)\ge \mathrm P(f)$.
 
 Equations (\ref{eq:811a}), (\ref{eq:812a}) follow from Lemmas \ref{lem:Q} and \ref{lem:P}.
 
Under the additional assumptions when $\varphi$ is $N$-function and $W(\infty) = \infty$, the first equation in (\ref{eq:813}) has been presented in Theorem 4.8 in \cite{KLR}.
 \end{proof}

Now let us summarize all known results describing the K\"othe dual  of the Orlicz-Lorentz space $\Lambda_{\varphi,w}$. For the space $\Lambda_{\varphi,w}$ by $\|\cdot\|_{\Lambda_{\varphi,w}}$ and  $\|\cdot\|^0_{\Lambda_{\varphi,w}}$ denote the  Luxemburg and Orlicz norm respectively. Recall that $\varphi_*(t)= \sup_{s\ge 0}\{st - \varphi(s)\}$, $t\ge 0$, is the complementary function to the Orlicz function $\varphi$. 

In the next theorem we state  complete descriptions of the dual spaces of the Orlicz-Lorentz space equipped with two standard Luxemburg and Orlicz norms. Recall indeed that the Orlicz-Lorentz space $\Lambda_{L_\ph,w} = \Lambda_{\ph,w}$ has a natural modular space structure given by the modular $\Phi$ defined in the equation (\ref{eq:modularOL}),
with respect to which $\Lambda_{\ph,w}$ is equipped with both a Luxemburg norm $\|\cdot\|_{\Lambda_{\varphi,w}}$ and an Orlicz norm $\|\cdot\|^0_{\Lambda_{\varphi,w}}$.  It is easy to see that these norms are identical to the norms of $\Lambda_{L_\ph,w} $ when the Orlicz space $L_\ph$ is equipped respectively with its own Luxemburg or Orlicz norm.

\begin{theorem}\label{th:dualOL} 
Let $\varphi$ be an Orlicz function and $w$ be a decreasing positive weight function on $I = (0,a)$, $a\le \infty$, such that $W<\infty$ on $I$. Then the K\"othe dual spaces to the Orlicz-Lorentz spaces $(\Lambda_{\varphi,w}, \|\cdot\|_{\Lambda_{\varphi,w}})$ and $(\Lambda_{\varphi,w}, \|\cdot\|^0_{\Lambda_{\varphi,w}})$ are as follows
\[
(\Lambda_{\varphi,w}, \|\cdot\|_{\Lambda_{\varphi,w}})' =
(\mathcal{M}_{\varphi_*,w}, \|\cdot\|^0_{\mathcal{M}_{\varphi_*,w}})\ \ \text{and}\ \ 
(\Lambda_{\varphi,w}, \|\cdot\|^0_{\Lambda_{\varphi,w}})' =
(\mathcal{M}_{\varphi_*,w}, \|\cdot\|_{\mathcal{M}_{\varphi_*,w}}),
\] 
where the norms $\|\cdot\|_{\mathcal{M}_{\varphi_*,w}}$ and $\|\cdot\|^0_{\mathcal{M}_{\varphi_*,w}}$ are given by (\ref{eq:811}) and (\ref{eq:812}),  respectively, where $\varphi$ is replaced by $\varphi_*$. 

\end{theorem}

 \begin{proof} This is is a consequence of Corollary \ref{cor:dual-lambda}, and the fact that when $E=L_\ph$ is an Orlicz space equipped with its Luxemburg (resp. Orlicz) norm then its K\"othe dual $E'$ is $L_{\ph_*}$ equipped with its Orlicz (resp. Luxemburg) norm.
\end{proof}

Comparing to Theorem 4.8 in \cite{KLR},  the above theorem is more general since it is proved here without  additional assumptions that $\varphi$ is $N$-function and $W(\infty) = \infty$. It is also more informative since it provides  three different formulas for the norms   in the dual space $\mathcal{M}_{\varphi_*,w}$. In fact each Luxemburg and Orlicz norm have   three  formulas expressed by (\ref{eq:811}) and (\ref{eq:812}),  corresponding either to modular $\rm Q$ or $\rm P$  or to level functions. The ones related to the modular $\mathrm Q$ are new here.

Finally we obtain a corollary on representation of the dual space for the classical Lorentz space $\Lambda_{p,w}$. If $\varphi(t) = t^p$, $1\le p<\infty$, then we use the following notations
\[
\Lambda_{p,w} := \Lambda_{\varphi,w} \ \ \ \text{and}\ \ \ \ \mathcal{M}_{p,w}: = \mathcal{M}_{\varphi,w}.
\]
In this case $\ph_*(t)=\frac{p^{1-q}}qt^q$ and the Orlicz norms on $L_{\ph_*}$ and  $\mathcal{M}_{\ph_*,w}$ coincide with the classical norms on $L_q$ and $\mathcal{M}_{q,w}$ respectively.  We provide below three different formulas of the norm in the dual space $(\Lambda_{p,w})^*$.  The  formula (8.14) has been presented as Corollary 4.9 in \cite{KLR},  and (8.15) has been proved by  Halperin in \cite[ Theorem 6.1,  Corollary, p. 288]{Ha53}. The first expression however, (8.13),  is new and it results from the introduction of the space $Q_{E,w}$ and $(\Lambda_{E,w})' = Q_{E',w}$.

\begin{theorem}\label{th:dualL} 
Let $1<p<\infty$, $\frac1p + \frac1q =1$, and $w$ be a decreasing positive weight function on $I = (0,a)$, $a\le \infty$, such that $W<\infty$ on $I$. Then
\[
(\Lambda_{p,w})'= \mathcal{M}_{q,w}.
\]
If in addition $W(\infty) = \infty$ when $I=(0,\infty)$, then
the dual space $(\Lambda_{p,w})^*$ is isometric to $\mathcal{M}_{q,w}$. In fact for every $F\in (\Lambda_{p,w})^*$ there exists $f\in \mathcal{M}_{q,w}$ such that
\[
F(g) = \int_I fg\,dm, \ \ \ g\in \Lambda_{p,w},
\]
and
\begin{align}
\|F\|=\|g\|_{\mathcal{M}_{q,w}}&= \inf \left\{\left(\int_I(g^*)^qw^{1-q}\right)^{1/q}: f\prec g\right\}\\
&= \inf \left\{\left(\int_I(f^*)^qv^{1-q}\right)^{1/q}: v\prec w, v>0, v\downarrow\right\}\\
&=\left(\int_I [(f^*)^0]^q w^{1-q}\right)^{1/q}.
\end{align}

\end{theorem}
\begin{proof}
The K\"othe duality follows from Theorem \ref{th:dualOL}. It is also well known and easy to show that $\Lambda_{p,w}$ is order continuous when $W(\infty) = \infty$ in the case of $I=(0,\infty)$. Therefore the K\"othe dual  space is isometric to the dual space via integral functionals \cite[Theorem 4.1]{BS}. 

\end{proof}

\begin{remark} 
 For $f\in L^0$ define
\[
 \mathrm Q_{\ph,w}(f)=\inf \left\{\int_I \ph(h) w\,dm: h\downarrow \hbox{ and }f\prec hw\right\}.
  \]
This formula was  introduced once by K. Nakamura \cite{Na}, who determined the modular dual to the natural modular in $\Lambda_{\ph,w}$ as being $Q_{\ph_*,w}$, when $\ph$ is a $N$-function satisfying a $\Delta_2$-condition and $\varphi_*$ a complementary function to $\varphi$. Note that 
\[Q_{\ph,w}(f)=\inf \left\{\mathrm M(g): f\prec g\hbox{ and }g/w\downarrow\right\}.
\]
Hence clearly $\mathrm Q(f)\le  Q_{\ph,w}(f)$. On the other hand since $f^0/w$ is non-increasing, we have $Q_{\ph,w}(f)\le M(f^0)=\mathrm Q(f)$ and finally $Q_{\ph,w}(f)=\mathrm Q(f)$.
\end{remark}

\section[-]{Examples of $M_{E,w}$ and $Q_{E,w}$ spaces}\label{sec:examples}
{}\medskip
 Let $w$ be a positive decreasing weight on $I=(0,a)$ such that $W<\infty$ on $I$, and $E$ be a Banach function space defined on the interval  $J=(0,b)$, $b=W(a)$, equipped with the Lebesgue  measure $m$. 
 In this section we will identify the spaces $M_{E,w}$ and $Q_{E,w}$ for some classical spaces $E$. Note that $M_{E,w},\, Q_{E,w}\subset L^0(I)$. 
 
\medskip

\begin{example}
 \emph{If  $E=L_1(J)$, then $(M_{L_1,w},\|\cdot\|_{M_{L_1,w}}) = (Q_{L_1,w}, \|\cdot\|_{Q_{L_1,w}}) = (L_1(I), \|\cdot\|_1)$.}
\end{example}

\begin{proof}
 Clearly $E_w = (L_1)_w=L_1(I,\omega)$ is a weighted $L_1$ space. 
We also have 
\[
f\in M_{L_1,w}\iff \frac{f^*}w\in (L_1)_w\iff \int_I \frac{f^*}w w\, dm<\infty\iff \int_I f^*\,dm<\infty \iff f\in L_1(I).
\]
Hence  $M_{L_1,w}=L_1(I)$ with the same norms. It follows that  $Q_{L_1,w}=L_1(I)$, also with the same norms.
\end{proof}

\medskip

\begin{example}\label{ex:9.2} \emph{If  $E=L_\infty(J)$ then 
\[
 M_{L_\infty,w} =\{f: \|f\|_{M_{L_\infty,w}}<\infty\} \ \  \text{ with } \ \ \ \|f\|_{M_{L_\infty,w}} = \inf\{C: f^* \le Cw\} =\|f^*/w\|_{L_\infty(I)}
 \]
\hfill $(Q_{L_\infty,w}, \|\cdot\|_{Q_{L_\infty,w}}) = (M_W, \|\cdot\|_{M_W}).$ \hfill\null
}
\end{example}

\begin{proof} 
 The weighted space $(L_\infty)_w$ consists of all  essentially bounded functions  on $I$ with respect to the measure  $d\omega = wdm$. Since $w$ is positive both spaces $(L_\infty)_w$ and $L_\infty(I)$ coincide with equality of norms. Thus 
$f\in M_{L_\infty,w}$ if and only if  $\frac{f^*}w\in (L_\infty)_w=L_\infty(I)$,  and  
\[
\|f\|_{M_{{L_\infty},w}} =\left\|\frac{f^*}w\right\|_{(L_\infty)_w} = \left\|\frac{f^*}w\right\|_{\infty}=\inf\{C: f^*\le Cw\}.
 \]
Note that  the gauge $\|\cdot\|_{M_{L_\infty,w}}$ is not a norm. As for the space $Q_{L_\infty,w}$,  by its definition 
$f\in Q_{L_\infty,w}$ if and only if there exists $g\in M_{L_\infty,w}$ with  $f\prec g $. This  is equivalent to
\begin{equation}\label{eq:9.1}
\exists g\in L^0_+,\, g\downarrow,   C>0 \hbox{ with }  g\le Cw \hbox{ and } \forall x\in I, \int_0^x f^*\,dm\le \int_0^x g\, dm.
\end{equation}
The above statement is equivalent to $ \int_0^x f^*\,dm\le C \int_0^x w\, dm$ for all $x\in I$ with the same constant $C$ as in (\ref{eq:9.1}). It follows that  $f\in M_W$ and 
\begin{align*}
\|f\|_{Q_{L_\infty,w}}& = \inf \{\|g\|_{M_{L_\infty,w}}: f\prec g, \ g \in {M_{L_\infty,w}} \} =  \inf \{C: f\prec g, \ g\le Cw\}  \\
&=\inf\{C: \  f\prec Cw\}= \sup_{x\in I}\frac 1{W(x)}\int_0^x f^*dm = \|f\|_{M_W}.
\end{align*}
Thus $Q_{L_\infty,w}$ coincides with the Marcinkiewicz space $M_W$ with the same norms.
\end{proof}

\medskip

\begin{example} \emph{ If  $E=L_1\cap L_\infty(J)$ then
\[
(M_{L_1 \cap L_\infty,w}, \|\cdot\|_{M_{L_1 \cap L_\infty,w}}) = (L_1 \cap M_{L_\infty,w}, \|\cdot\|_{L_1 \cap M_{L_\infty,w}})
\]
\[
(Q_{L_1 \cap L_\infty,w}, \|\cdot\|_{Q_{L_1 \cap L_\infty,w}}) =( L_1 \cap M_W, \|\cdot\|_{ L_1 \cap M_W}).
\]}
\end{example}

\begin{proof}

Let  $f\in M_{L_1\cap L_\infty,w}$. Then by Proposition \ref{properties class M} 
\begin{align*}
\|f\|_{M_{L_1\cap L_\infty,w}} & = 
\left\|\frac{f^*}{w}\right\|_{(L_1 \cap L_\infty)_w}=
\left\|\frac{f^*}{w}\circ W^{-1}\right\|_{L_1 \cap L_\infty(J)}\\
&=\left\|\frac{f^*}{w}\circ W^{-1}\right\|_{L_1(J)} \vee \left\|\frac{f^*}{w}\circ W^{-1}\right\|_{L_\infty(J)}\\&=
\|f\|_{L_1(I)}\vee \left\|\frac{f^*}{w}\right\|_{(L_\infty)_w} 
= \|f\|_{L_1(I)}\vee \|f\|_{M_{L_\infty,w}}
=\|f\|_{{L_1} \cap M_{L_\infty,w}}.
\end{align*} 
Thus  $M_{L_1\cap L_\infty,w} = L_1\cap M_{L_\infty,w}$ with identical gauges. \smallskip

For every $g\in L_1\cap L_{\infty,w} $ and $f\prec g$ we have $\|f\|_1 \le \|g\|_1$, and 
\[\|f \|_{M_W} = \inf\{C: f\prec Cw\} \le \inf\{C: g^* \le Cw\} = \|g\|_{M_{L_\infty,w}}.\]
Thus
\[
\|f\|_{L_1\cap M_W} = \|f\|_1 \vee \|f\|_{M_W}\le \|g\|_1 \vee \|g\|_{M_{L_\infty,w}} =\|g\|_{L_1\cap {M_{L_\infty,w}}}.
\]
It follows that $Q_{L_1\cap L_\infty,w} \subset L_1 \cap M_W$, and that for every $f\in Q_{L_1\cap L_{\infty,w}}$,
\[
\|f\|_{L_1\cap M_W} \le \|f\|_{Q_{L_1\cap L_\infty,w}}.
\]

 Conversely if $f\in L_1\cap M_W$ then $f^*\in L_1$ and $f\prec Cw$ where $C=\|f\|_{M_W}$. Then for every $x\in I$ we have
\[
\int_0^x f^*dm \le CW(x)\wedge \|f\|_{1}.
 \]
We have $b= W(a) = \sup_{x\in I}W(x)$. If $Cb\le \|f\|_1$ then the preceding inequalities mean that $f\prec Cw$. Setting $g=Cw$ we have $\|g\|_1\le \|f\|_1$ and $\|g\|_{M_{L_\infty,w}}=C=\|f\|_{M_W}$, hence $g\in L_1\cap M_{L_{\infty,w}}= M_{L_1\cap L_\infty,w}$ and by $f\prec g$ it follows that $f\in Q_{L_1\cap L_\infty,w}$ with
\[ 
\|f\|_{Q_{L_1\cap L_\infty,w}}\le \|g\|_{M_{L_1\cap L_\infty,w}} \le \|f\|_1\vee \|f\|_{M_W}.
\]
On the other hand if  $Cb> \|f\|_1$, there exists $x_f\in I$ such that $CW(x_f)=\|f\|_1$. Setting now $g=Cw\chi_{(0,x_f)}$, observe that $f\prec g$, $\|g\|_1= \|f\|_1$ and $\|g\|_{M_{{L_\infty,w}}}=C=\|f\|_{M_W}$, and conclude as above.
\end{proof}

\begin{example} \emph{If  $E=L_1 + L_\infty(J)$ then}
\begin{align*}
(M_{L_1 + L_\infty,w}, \|\cdot\|_{M_{L_1 + L_\infty,w}}) &= ({L_1+M_{L_{\infty},w}}, \|\cdot\|_{L_1+M_{L_{\infty},w}}),
\\
(Q_{L_1+L_\infty,w}, \|\cdot\|_{Q_{L_1+L_\infty,w}}) &= (L_1 + M_W, \|\cdot\|_{L_1 + M_W}).
\end{align*}
\end{example}

\begin{proof}
 We can show directly that $(L_1 + L_\infty)_w = (L_1)_w + (L_\infty)_w= (L_1)_w +   L_\infty(I)$ with equality of norms. By Example \ref{ex:9.2}, a function $v$ belongs to $M_{L_\infty,w}$ if there exists $C>0$ such that $v^* \le Cw$. Thus 
 \reqnomode
\begin{align}\label{decomp}
& f\in M_{L_1+L_\infty,w} \hbox{ with } \|f\|_{M_{L_1+L_\infty,w}}<1 \notag\\
&\iff \exists\, g\in (L_1)_w,h\in  L_\infty(I): \frac{f^*}w = g+h, \|g\|_{(L_1)_w}+\|h\|_\infty<1\notag\\
&\iff \exists  u\in   L_1(I), v\in L^0, C\ge 0: f^*=u+v,  |v|\le Cw, \|u\|_1+C<1 
\hfill{}\\ &\implies \exists  u\in   L_1(I), v\in L^0, C\ge 0: f^*=u+v,  v^*\le Cw, \|u\|_1+C<1\notag\\
&\iff  f^*\in L_1+M_{L_{\infty},w} \hbox{ with } \|f^*\|_{L_1+M_{L_\infty,w}}<1.\notag
\end{align}
We want to prove that $\|f^*\|_{L_1+M_{L_\infty,w}}<1$ implies $\|f\|_{L_1+M_{L_\infty,w}}<1$. Let $f^*=u+v$ with $v^*\le Cw$ and $\|u\|_1+C<1$. Let us consider two cases. 
\medskip

 Assume first that either the interval $I$ is finite or $\lim\limits_{t\to\infty}f^*(t)=0$. Then by Proposition \ref{prop:ryff} there exists a measure preserving transformation $\sigma$, either from $I$ onto $I$ if the support of $f$ has finite measure or from the support of $f$ onto  $I$ if the support of $f$ has infinite measure, such that $f=f^*\circ \sigma$. Then $f=u\circ\sigma+ v\circ\sigma$ with $u\circ\sigma, v\circ\sigma$ equimeasurable with $u,v$ respectively. In particular $\|u\circ\sigma\|_1=\|u\|_1$ and $(v\circ\sigma)^*=v^*\le Cw$. Hence  $\|f\|_{L_1+M_{L_\infty,w}}<1$.
\medskip

The proof is similar  if $I$ is infinite and $\lim\limits_{t\to\infty}f^*(t)>0$,  by using now Lemma \ref{lem:Ryff-3}.
\medskip

Therefore we have shown  that  if $\|f\|_{M_{L_1+L_\infty,w}}<1$ then $\|f\|_{L_1+M_{L_\infty,w}}<1$, which implies $M_{L_1+L_\infty,w}\subset L_1+M_{L_{\infty},w}$ and that this inclusion is gauge-decreasing. 

 Let us prove now the converse inclusion.
Let $f\in L_1+M_{L_{\infty},w}$ with $\|f\|_{L_1+M_{L_{\infty},w}}<1$ and $f=u+v$ be a decomposition with $u\in  L_1(I), v^*\le Cw$ and $\|u\|_1+C<1$. By the Lorentz-Shimogaki inequality \cite[Chapter 3, Theorem 7.4]{BS},
\[ f^*-v^*\prec (f-v)^*=u^*\]
it follows that $u_1:= f^*-v^*\in L_1$, with $\|u_1\|_1\le \|u\|_1$. Thus
\[
f^*= u_1+v^*, \hbox{ with } v^*\le Cw \hbox{ and } \|u_1\|_1+C<1.
\]
By (\ref{decomp}) it means that $\|f\|_{M_{L_1+L_\infty,w}}<1$.   Hence $L_1+M_{L_{\infty},w} \subset M_{L_1+L_\infty,w}$, and finally $L_1+M_{L_{\infty},w} = M_{L_1+L_\infty,w}$ with equal norms.
\medskip

Now we will show that $Q_{L_1+L_\infty,w}= L_1 + M_W$ with equality of  norms. First,
\begin{align*}
f\in Q_{L_1+L_\infty,w} &\iff \exists\, g\in M_{L_1+L_\infty,w},\ f\prec g\\
&\iff \exists\, g\in {L_1}+
M_{L_\infty,w},\ f\prec g\\
&\iff \exists u\in  L_1, \exists v \in M_{L_\infty,w}: f\prec u+v \\
&\implies \exists u\in  L_1, \exists v\in M_{L_\infty,w}: f^*\prec u^*+v^*.
\end{align*}
Since $f^*\prec u^*+v^*$, then by Fact \ref{fact:decomp} there exists a decomposition $f=u'+v'$ with $u'\prec u^*$, $v'\prec v^*$. Then from $u\in L_1$ and $v\in M_{L_\infty,w}$ it follows that 
\[
u'\in L_1, \ v'\in M_W \ \  \text{and} \ \ \|f\|_{L_1 + M_W} \le \|u'\|_{1}+\|v'\|_{M_W} \le \|u^*\|_{1}+\|v^*\|_{M_{L_\infty,w}}.
\]
 Thus $ Q_{L_1+L_\infty,w}\subset L_1+M_W$.

Moreover, in view of the above paragraph if $\|f\|_{Q_{L_1 + L_\infty,w}}<1$ we may choose $g\in M_{L_1+L_\infty,w} = L_1 + M_{L_\infty,w}$ with $f\prec g$ and $\|g\|_{L_1+M_{L_\infty,w}}<1$.  Thus there is a   decomposition $g=u+v$ with $\|u\|_{1}+\|v\|_{M_{L_\infty,w}}<1$. Hence   $\|f\|_{L_1+M_W}\le \|g\|_{L_1 + M_W} \le \|u\|_1 + \|v\|_{M_W}\le \|u\|_1 + \|v\|_{M_{L_\infty,w}} <1$. Hence  $\|f\|_{L_1 +M_W} \le \|f\|_{Q_{L_1 + L_{\infty,w}}}$.

As for the converse direction, let $f\in L_1+M_W$ with $\|f\|_{L_1+M_W}<1$ and $f=k+h$ be a decomposition with $\|k\|_{1}+\|h\|_{M_W}<1$. Then
\[f^*\prec k^*+h^*\prec k^*+Cw\ \hbox{ with } C=\|h^*\|_{M_W}=\|h\|_{M_W}.
\]
Setting $g=k^*+Cw$, we have $f\prec g$, and $\|Cw\|_{M_{L_\infty,w}} = C \|w/w\|_{(L_\infty)_w} = \|h\|_{M_W}$. Hence 
 \[
  \|g\|_{L_1+M_{L_\infty,w}}\le \|k\|_1 + \|Cw\|_{M_{L_\infty,w}} = \|k\|_1 + \|h\|_{M_W} < 1.
  \]
    Since $g$ is decreasing and $\|\cdot\|_{L_1+M_{L_\infty,w}} =\|\cdot\|_{M_{L_1+L_\infty, w}}$ , we have $g=g^*\in M_{L_1+L_\infty, w}$ with $\|g\|_{M_{L_1+L_\infty, w}}<1$, hence $f\in Q_{M_{L_1+L_\infty, w}}$ with $\|f\|_{Q_{M_{L_1+L_\infty,w}}}<1$. Thus $ L_1+M_W\subset Q_{M_{L_1+L_\infty,w}}$ and 
$\|f\|_{L_1 +M_W} \ge \|f\|_{Q_{L_1 + L_{\infty,w}}}$.    
    
    Consequently $ L_1+M_W = Q_{M_{L_1+L_\infty,w}}$ with equality of norms.
\end{proof}

\section*{Appendix}

We give here a self-contained and simple proof of Fact \ref{lem:submaj-phi}.
\medskip

\noindent For $x,y>0$ set
\[D(x,y)= 
\begin{cases}
{\psi(x)-\psi(y)\over x-y} &\hbox{ if } x\ne y,\\
\psi'_+(x) &\hbox{ if } x= y,
\end{cases}
\]
where $\psi'_+$ is the right derivative of $\psi$. 
Observe that if $x_1\le x_2$ and $y_1\le y_2$ then $D(x_1,y_1)\le D(x_2,y_2)$, by convexity of the function $\psi$. Indeed if we set $a_i=\min(x_i,y_i)$ and $b_i=\max(x_i,y_i)$, $i=1,2$, then $a_1\le a_2$, $b_1\le b_2$ and $D(x_1,y_1)=D(a_1,b_1)$ and $D(x_2,y_2)=D(a_2,b_2)$ are the respective  slopes of the chords  of the graph of $\psi$ corresponding to the intervals $[a_1,b_1]$ and $[a_2,b_2]$, or the  slopes of the right tangent  lines in the case $a_i=b_i$, $i=1,2$.

Since for any $f\in L^0(\Omega)$, $f^{*,\mu}\in L^0(0,\mu(\Omega))$, we may assume without loss of generality that $f,g:[0,\infty)\to [0,\infty)$ are decreasing non-negative functions.  Then  the function $D(f,g): t\mapsto D(f(t),g(t))$ is also decreasing. Assuming that $f\prec g$, we want to show that $\psi(f)\prec \psi(g)$.  We note that for $x\ge 0$,
\begin{align*}
0\le \int_0^x g\,dm-\int_0^x\,fdm=\int_0^x(g-f)\,dm=\int_0^x(g-f)_+\,dm-\int_0^x(g-f)_-\,dm,
\end{align*}
thus  for $x\ge 0$,
\[
\int_0^x(g-f)_-\,dm\le \int_0^x(g-f)_+\,dm.
\]
By Hardy's Lemma \cite[Proposition 3.6, Ch.2]{BS} this implies that 
\[
\int_0^x(g-f)_-D(g,f)\,dm\le \int_0^x(g-f)_+D(g,f)\,dm
\]
for $x\ge 0$. We have $\psi(g)-\psi(f)=(g-f)D(g,f)$ and since $D(g,f)\ge 0$ it follows that $(\psi(g)-\psi(f))_\pm=(g-f)_\pm D(g,f)$. Hence the preceding inequality may be rewritten as
\begin{align}\label{eq:submaj}
 \int_0^x(\psi(g)-\psi(f))_-\,dm\le \int_0^x(\psi(g)-\psi(f))_+\,dm.
\end{align}
Supposing that $\psi(f)$ is integrable on finite intervals, it implies the same for $(\psi(g)-\psi(f))_-$ since $(\psi(g)-\psi(f))_-\le \psi(f)$. Then for any $x>0$,
\begin{align*}
\int_0^x\psi(g)\,dm-\int_0^x\psi(f)\,dm&=\int_0^x(\psi(g)-\psi(f))\,dm\\
&= \int_0^x(\psi(g)-\psi(f))_+\,dm-\int_0^x(\psi(g)-\psi(f))_-\,dm\ge 0,
\end{align*}
which implies that  $\psi(f)\prec \psi(g)$. If we have no information on the local integrability of $\psi(f)$, we may apply the above to the couple $(f\wedge n,g)$, where $n\in \mathbb{N}$. Indeed we have $f\wedge n\le f\prec g$, $f\wedge n$ is decreasing, and $\psi(f\wedge n)=\psi(f)\wedge \psi(n)$ is bounded, and thus integrable on finite  intervals. Hence for all $n\in \mathbb{N}$, $x\ge 0$,
\[
\int_0^x \psi(f)\wedge\psi(n)\,dm\le \int_0^x \psi(g)\,dm, 
\]
and passing to the limit $n\to\infty$ we obtain that $\psi(f)\prec\psi(g)$.
\hfill\qed


\end{document}